\documentclass{article}
\usepackage{amsmath,amsfonts,amssymb,amsthm,fullpage}
\usepackage{subfigure}
\usepackage{tikz,pgfplots}
\usetikzlibrary{intersections,calc}

\newtheorem{lemma}{Lemma}[section]

\newtheorem{theorem}[lemma]{Theorem}
\newtheorem{proposition}[lemma]{Proposition}
\newtheorem{definition}[lemma]{Definition}

\newtheorem{remark}{Remark}

\newcommand{\N}{\mathbb{N}} 
\newcommand{\R}{\mathbb{R}}

\newcommand{\dt}{{\delta t}} 
\newcommand{\dx}{{\delta x}}
\newcommand{\CFL}{{\rm CFL}}
\newcommand{\dd}{{\mathrm{d}}}
\renewcommand{\t}{{\mathsf{t}}}
\newcommand{\normetriple}[1]{\vert\hspace{-0.3 mm}\vert\hspace{-0.3 mm}\vert #1 \vert\hspace{-0.3 mm}\vert\hspace{-0.3 mm}\vert}

\begin{document}

\title{Asymptotic behavior of splitting schemes involving
time-subcycling techniques}


\author{%
{\sc Guillaume Dujardin\thanks{Email: guillaume.dujardin@inria.fr}}\\[2 pt]
Inria Lille Nord-Europe, \'Equipe M\'EPHYSTO, and \\
Laboratoire Paul Painlev\'e, Universit\'e de Lille, CNRS UMR 8525\\[6 pt]
{\sc and}\\[6 pt]
{\sc Pauline Lafitte\thanks{Corresponding author.
Email: pauline.lafitte@centralesupelec.fr}}\\[2 pt]
CentraleSup\'elec, Lab. MAS, and \\
F\'ed\'eration de Math\'ematiques, FR CNRS 3487}

\maketitle
 
\bigskip 
\begin{abstract}
  {This paper deals with the numerical integration of well-posed multiscale
  systems of ODEs or evolutionary PDEs. As these systems 
  appear naturally in engineering problems,
  time-subcycling techniques are widely used every day to improve
  computational efficiency.
  These methods rely on a decomposition of the vector field in a fast part
  and a slow part and take advantage of that decomposition.
  This way, if an unconditionnally stable (semi-)implicit scheme cannot be easily implemented, one can integrate the fast equations with a much smaller time
  step than that of the slow equations, instead of having to integrate the
  whole system with a very small time-step to ensure stability.
  Then, one can build a numerical integrator using a standard composition
  method, such as a Lie or a Strang formula for example.
  Such methods are primarily designed to be convergent in short-time
  to the solution of the original problems.
  However, their longtime behavior rises interesting questions, the answers
  to which are not very well known.
  In particular, when the solutions of the problems converge in time to an
  asymptotic equilibrium state, the question of the asymptotic accuracy
  of the numerical longtime limit of the schemes
  as well as that of the rate of convergence
  is certainly of interest.
  In this context, the asymptotic error is defined as the difference
  between the exact and numerical asymptotic states.
  The goal of this paper is to apply that kind of numerical methods
  based on splitting schemes with subcycling to some simple
  examples of evolutionary ODEs and PDEs that have attractive
  equilibrium states, to address the aforementioned questions
  of asymptotic accuracy, to perform a rigorous analysis, and to compare
  them with their counterparts without subcycling.
  Our analysis is developed on simple linear ODE and PDE toy-models and
  is illustrated with several numerical experiments on these toy-models
  as well as on more complex systems.}
  {Lie and Strang splitting schemes - Subcycling - $\theta$-schemes - Longtime
  asymptotics - Asymptotic error - Asymptotic order
  }
\end{abstract}

\section{Introduction}
\label{sec:Intro}
Time-subcycling is a way to speed up numerical computations for an
evolutionary multiscale problem by splitting the underlying operator
and treating its different parts with adapted time-steps to build
up a numerical integrator {that is less costly}. {The idea
is to split the problem into subproblems, each with an identified timescale.
Let $\dt>0$ 
be a discretization time-step ensuring that the subproblem with
the larger time scale is stable.
Let $N>>1$ be the ratio of the time scales the two subproblems.
A subcycling technique consists in iterating $N$ times a scheme for the second subproblem with time-step $\dt/N$ so that the time matches $\dt$. We emphasize that these methods are very useful if an unconditionnally stable (semi-)implicit scheme is difficult or costly to implement.}

The analysis of these methods over finite time intervals is rather
similar to that of composition methods over finite time intervals; see \cite{mll}.
In contrast, our aim {in this paper}
is to determine how well the subcycling techniques capture
the right asymptotic state for continuous dynamical systems described by ODEs or
PDEs, the solutions of which converge to a steady state as time goes to
infinity. If there are no unconditionnally stable (semi-)implicit schemes at hand, in order to save computational time, the subcycling techniques have
been very widely used for schemes associated with multiscale systems, which
have (at least) one component that has to be computed through an explicit scheme
{and are therefore} constrained by a limitation of the time-step
(CFL); see \cite{bgmps,glg,cgl}.
Related local time-stepping techniques
have been developed extensively for multiscale problems arising in computational
fluid and structural dynamics; see \cite{p,wd1,wd2}. The simulation of transport or
diffusive phenomena in the presence of complex geometries requires local mesh
refinement, which imposes the use of finite element or discontinuous Galerkin
methods. An ever larger number of steps is needed if the chosen scheme is explicit,
due to the CFL condition,
or the inversion of large matrices if an implicit scheme
is preferred in order to alleviate the time-step restriction. The local
convergence of these methods has been established in a variety of cases (see
\cite{mg1,mg2,mg3} and references therein).

Splitting methods have been used
in several other applications where the use of subcyling techniques
was indeed crucial.
Let us emphasize at least
the numerical integration of the Community Atmosphere Model
(CAM) and the ENZO code for astrophysics.
The CAM is a global atmosphere model developed
at the US National Center for Atmospheric Research (NCAR) for the weather and
climate research communities; see \cite{taylor2010subcycled}.
ENZO is an open-source code developed in the US
for modeling astrophysical fluid flows
which involves adaptive mesh refinement and subcycling techniques; see
\cite{bryan2014enzo}.
In the references above, the gain obtained by using subcycling techniques
for these large multiscale problems is strongly emphasized.

The applications we have
specifically in mind are related to the recent development of the
``asym\-ptotic-preserving'' schemes in the sense of  \cite{j1999,j2010} for
kinetic equations. Schemes obtained using splitting techniques making use of suitable time scales
{were}  indeed proved
efficient for Boltzmann-type and Fokker-Planck equations by way of micro-macro
decompositions; see \cite{glg,lm,cglv,cgl}. However, if subcycling techniques have
been used in several test-cases, to our knowledge, the asymptotic error
between the exact and numerical longtime solutions has never been precisely
analyzed.

{The long term goal of our work is to be able to study the longtime
convergence (error estimates and rate of convergence) of subcycled schemes and
to compare it to that of non-subcycled schemes.
In this paper, we propose techniques
to formalize the rigorous analysis of the longtime convergence.
}
In particular, these techniques lead to the remarkable and unexpected asymptotic behavior of some Strang
splitting schemes, which approximate better the solution in longtime than
locally predicted, in the spirit of the asymptotic high-order schemes developed
by \cite{abn}.

{We formulate the question of the longtime convergence of numerical methods with or without subcycling in a generic framework of
(partial) differential problems, which would be large enough to include
interesting applications. However, tackling the question in full generality
would probably lead to a too abstract and technical work, so
we develop our analysis on several simple examples (simple systems of ODEs and PDEs) which write as
autonomous Cauchy problems of order one in time
with a fast and a slow component in the vector field. The common feature
of the examples we consider is the existence of a stationary state to which
the solutions converge exponentially fast
in longtime.
}

For every example, we introduce several schemes, with and without subcycling,
we perform numerical experiments on the longtime behavior of the proposed
schemes, and we provide the reader with a mathematical analysis
of the numerical results.
This paper is organized as follows. In Section \ref{sec:fw}, we introduce the general
differential framework (ODEs and PDEs systems) together with the numerical splitting methods
with or without subcycling under consideration. We introduce the concepts of asymptotic error and asymptotic order. Also, we express the local order of a splitting scheme, with or without subcycling, as a function of the order of the underlying schemes and of the order of the splitting method. This result is related to the previous work by \cite{CS2008}.
We then perform our analysis on several toy-models 
in the remaining sections.
Sections \ref{sec:linearanalysis} and \ref{sec:nonlinearanalysis}
are devoted to two different
examples (one linear in Section \ref{sec:linearanalysis}  and one nonqlinear in Section \ref{sec:nonlinearanalysis}) of differential systems with two different time scales. The choice of the examples is 
strongly inspired by the analysis of the Dahlquist test equation
when studying the asymptotic stability of schemes for stiff ODEs (see \cite{hw}) and
of the analysis led by \cite{temam}.
Both systems have exact and explicit solutions so one can do any computations
and estimates involving the exact flows.
We prove properties
about the asymptotic orders of the schemes for the linear example (see
Propositions~\ref{prop:LiesansCsomos}, \ref{prop:CsomossansLie}
and \ref{prop:Strang}) which are illustrated
by several numerical experiments in the nonlinear case in Section \ref{sec:nonlinearanalysis}.
We comment on the differences between schemes with and without subcycling.
In Section \ref{sec:reactiondiffusion}, we perform the same kind of analysis
for a 1D linear coupled reaction-diffusion
system.
For this problem, the boundary conditions play a crucial role in the
existence of attractive equilibrium states. We focus on two cases of
boundary conditions (homogeneous and inhomogeneous Dirichlet conditions).
For homogeneous boundary conditions, we introduce a subcycled Lie-splitting
scheme, we address the question of its rate of convergence towards
the equilibrium state (see Theorem~\ref{decroissancenum})
and we compare this rate to that of the exact solution
(see Theorem~\ref{th:SolEx}).
For inhomogeneous Dirichlet boundary conditions, we compare several
splitting schemes with and without subcycling and we address
the question of the asymptotic error which depends on both the time
and space discretization parameters.
For the subcycled Lie-splitting scheme, we prove that the asymptotic
equilibrium state of the scheme is a uniform-in-$\dt$ second order
${\rm L}^2$-approximation of the exact asymptotic equilibrium state under
a CFL-like condition (see Theorem~\ref{Th:ApproxLieDirInHomo}).
We illustrate numerically the asymptotic behavior of the Strang
schemes and the weighted splitting schemes introduced by \cite{Csomos05}.

\section{General framework and definition of the asymptotic error}\label{sec:fw}

\subsection{General framework}
\label{subsec:GeneralFramework}
The framework which is under scrutiny in this paper is the study
of numerical approximations of multiscale systems of ODEs, and, more generally,
of PDEs.

We are interested in globally well-posed Cauchy problems that write
\begin{equation}\label{eq:ms}
\begin{cases}
\dfrac{\mathrm d}{{\mathrm d}t}W(t)=\mathsf{f}(W(t)),\quad t>0\\
W(0)=W^0,
\end{cases}
\end{equation}
where $\mathsf{f}:D(\mathsf{f})\subset X\rightarrow X$ is not necessarily
linear, $X$ is a Banach space and $W^0\in D(\mathsf{f})$.
We assume that the solution to \eqref{eq:ms} is given by a semi-group.
Fairly general sufficient conditions to ensure this property can be found
for example in \cite{CL71}.
We assume that there exists an asymptotically 
stable state $W^\infty_{\rm ex}\in X$ to which $W(t)$ converges as $t$
goes to infinity. In addition, we suppose that 
the system has a multiscale property:  the vector field can be split
into a fast and a slow part,
once the system is recast in a dimensionless form. We rewrite \eqref{eq:ms}  as
\begin{equation}\label{eq:mss}
\begin{cases}
\dfrac{\mathrm d}{{\mathrm d}t}W(t)=\dfrac{T_{\rm obs}}{T_{s}}\mathsf{f}_{s}(W(t))+\dfrac{T_{\rm obs}}{T_{f}}\mathsf{f}_{f}(W(t)),\quad t>0\\
W(0)=W^0,
\end{cases}
\end{equation}
where $T_{\rm obs}>0$ is an observation time, $T_{s}$ (resp. $T_{f}$) is the characteristic time of the slow (resp. fast) phenomenon and $\mathsf{f}_{s}$ (resp. $\mathsf{f}_{f}$) is the vector field corresponding to the slow (resp. fast) phenomenon, and $T_{s}$ is very large compared to $T_{f}$: $T_{s}/T_{f}=N>>1$.

\subsection{The concept of asymptotic error}
\label{sec:conceptasympterror}

Let $\dt>0$ be a fixed time-step and denote by $G(\dt):X\rightarrow X$ a numerical scheme for \eqref{eq:ms}, that provides
a numerical solution $(W^n)_{n\geq0}$ defined as $W^{n+1}=G(\dt)[W^n]$ for all $n\geq0$.
Assuming that the numerical scheme $G(\dt)$ has an asymptotically stable state, we define
\begin{equation}
W_{\rm num}^\infty=\lim_{n\rightarrow\infty}W^n.\end{equation}
Of course, $W_{\rm num}^\infty$ depends on $\dt$. 
\begin{definition}\label{def:monique}
We define the asymptotic error 
of the scheme $G(\dt)$ as
\begin{equation*}
\varepsilon^{\rm as} = W^\infty_{\rm num}-W^\infty_{\rm ex}. \end{equation*}
We say that
the asymptotic order (A-order) is at least $p\in\N^\star$ if
when $\dt$ tends to $0$, we have
\begin{equation*}
   \varepsilon^{\rm as}={\cal O}(\dt^p).
\end{equation*}
As usual, the A-order is the supremum of the set of
such $p$.
\end{definition}

\subsection{Splitting methods with and without subcycling}

We are interested in solving numerically problems of the form \eqref{eq:mss}
using splitting methods adapted to the slow/fast decomposition of the vector
field. More precisely,
we aim at studying the asymptotic error of splitting 
methods involving subcycling. In this context, using subcycling consists
in using a splitting method with different time-steps for the slow and fast
components. Taking $T_{\rm obs}=T_s$ in System \eqref{eq:mss}, we obtain
\begin{equation}
\label{eq:introsplit}
\begin{cases}
\dfrac{\mathrm d}{{\mathrm d}t}W(t)=\mathsf{f}_{s}(W(t))+N\mathsf{f}_{f}(W(t)),\quad t>0\\
W(0)=W^0.
\end{cases}
\end{equation}
Let us denote by $\Phi_f(\dt)$ (resp. $\Phi_s(\dt)$)
an approximation of the exact flow $\varphi_f(\dt)$ (resp. $\varphi_s(\dt)$) of
\begin{equation}
\label{eq:equationssplitees}
  \dfrac{\mathrm d}{{\mathrm d}t}W(t)=N\mathsf{f}_{f}(W(t))
  \qquad \left(\text{resp.}\ 
  \dfrac{\mathrm d}{{\mathrm d}t}W(t)=\mathsf{f}_{s}(W(t)) \right).
\end{equation}
In particular, we assume that both equations \eqref{eq:equationssplitees}
are solved by semi-groups with compatible domains, just as we did
for the global problem \eqref{eq:ms} in Section \ref{subsec:GeneralFramework}.
A classical splitting (or composition) method consists in setting
\begin{equation}
\label{eq:composition}
    \Phi_{\rm c} (\dt) = \Pi_{i=1}^k (\Phi_{s}(b_i\dt)\circ \Phi_{f}(a_i \dt )),
  \end{equation}
for some real coefficients $a_1,\dots,a_k,b_1,\dots,b_k\in\R$
and considering $\Phi_{\rm c}(\dt)$ as an approximation of the exact flow
$\varphi_{{\rm ex}}(\dt)$
of \eqref{eq:introsplit} on a time interval of size $\dt$.
A splitting method with subcycling
consists in taking
\begin{equation}
\label{eq:methodesplittingsc}
    \Phi_{\rm sc} (\dt) = \Pi_{i=1}^k \left(\Phi_{s}(b_i\dt)\circ \left(\Phi_{f}(a_i \dt/N)\right)^N\right),
\end{equation}
as an approximation of the same exact flow.

Since the analysis of the asymptotic error of splitting methods
with or without subcycling in such a general framework is
out of reach for the authors, we rather perform our analysis
on several examples.
These examples are linear ODEs (Section \ref{sec:linearanalysis}),
nonlinear ODEs (Section \ref{sec:nonlinearanalysis}) and
linear PDEs (Section \ref{sec:reactiondiffusion}) that
have the multiscale property detailed above.
Moreover, they allow us to perform an analysis in full detail.

\subsection{Local order of splitting methods with and without
subcycling}
We prove a somehow classical result
expressing
the local order of a splitting
scheme (with or without subcycling) as a function of the order of the
underlying schemes and the order of the splitting method in the context
of ODEs ({\it i.e. } $X=\R^d$ for some $d\in\N^\star$).

\begin{theorem}
\label{th:ordrenonlineaire}
Assume $X=\R^d$ for some $d\in\N^\star$.
Let us consider a differential system of the form \eqref{eq:introsplit}
with $D(\mathsf{f})=X$.
With the notations introduced above, we assume that
$\Phi_{f}(\dt)$ and $\Phi_{s}(\dt)$
are numerical methods of respective orders $p_f$ and $p_s$.
Moreover, we assume that a splitting method $\Phi_{\rm c}(\dt)$
is defined for some $a_1,\dots,a_n,b_1,\dots,b_n\in\R$
by the formula \eqref{eq:composition}
so that this method with the exact flows ($\varphi_f(\dt)$ and $\varphi_s(\dt)$)
has order $p_{\rm ex}$.
Then the order of the method $\Phi_{\rm c}(\dt)$ is at least
${\rm min}(p_f,p_s,p_{\rm ex})$, and so is the order of the method with subcycling
$\Phi_{\rm sc}(\dt)$ defined by Formula \eqref{eq:methodesplittingsc}.
\end{theorem}

\begin{proof}
  Since the methods $\Phi_{s}(\dt)$ and $\Phi_{f}(\dt/N)$ have orders $p_s$
and $p_f$ respectively, we may write, when $\dt \rightarrow 0$,
  \begin{equation*}
    \Phi_{s}(\dt) = \varphi_{s}(\dt) + {\mathcal O} (\dt^{p_s+1}) \quad\mbox{and}\quad
    \Phi_{f}(\dt/N) = \varphi_{f}(\dt/N) + {\mathcal O} (\dt^{p_f+1}).
  \end{equation*}
  The smoothness of the propagators implies that for all $j\in\N^\star$,
  \begin{equation*}
    \Phi^{j}_{f}(\dt/N) = \varphi^{j}_{f}(\dt/N) + {\mathcal O} (\dt^{p_f+1}),
  \end{equation*}
  where the constant in the Landau symbol depends on $j$. In particular,
for $j=N$, using the semi-group property of the exact flow, we have
  \begin{equation*}
    \Phi^N_{f}(\dt/N) = \varphi_{f}(\dt) + {\mathcal O} (\dt^{p_f+1}).
  \end{equation*}
  This implies
  \begin{align*}
    \Phi_{\rm sc}(\dt) = \Pi_{i=1}^n (\Phi_{s}(b_i\dt)\circ (\Phi_{f}(a_i\dt/N))^N) & =  
    \Pi_{i=1}^n 
    (\varphi_{s}(b_i\dt) + {\mathcal O} (\dt^{p_s+1}))\circ
    (\varphi_{f}(a_i\dt) + {\mathcal O} (\dt^{p_f+1}))
    \\
    & =  \Pi_{i=1}^n 
    (\varphi_{s}(b_i\dt)\circ \varphi_{f}(a_i\dt))
    + {\mathcal O} (\dt^{{\rm min}(p_f,p_s)+1}) \\
    & =  \varphi_{{\rm ex}}(\dt) + {\mathcal O}(\dt^{{\rm min}(p_f,p_s,p_{\rm ex})+1}),
  \end{align*}
since the splitting method $\Phi_c(\dt)$ is assumed to have order
$p_{\rm ex}$ when used with
the exact flows. This proves the result for $\Phi_{\rm sc}(\dt)$.
The proof for $\Phi_{\rm c}(\dt)$ is even simpler since there
is no need to compute the internal composition step.
\end{proof}

\section{Full analysis of the asymptotic error of splitting schemes
applied to a linear toy-model}
\label{sec:linearanalysis}
As a first example of system of the form \eqref{eq:introsplit},
we consider in this section the following example:
\begin{equation}
  \label{eq:cas-simple-lineaire}
  \begin{cases} 
    u'=-Nc(u-v)\\
v'=c(u-v),
  \end{cases}
\end{equation}
where $c> 0$ and $N\in\mathbb{N}$, with $N$ being 
large: it is the stiffness parameter in the problem.
From the dimensional viewpoint, $c$ is the inverse of a
characteristic time.
With the notations of Section \ref{sec:Intro}, we have
\begin{equation*}
X=\R^2,\quad W=\left(\begin{matrix}u\\ v\end{matrix}\right),\quad
{\mathsf f}_{\rm s} \left(\begin{matrix}u\\ v\end{matrix}\right) =
\left(\begin{matrix}0\\ c(u-v)\end{matrix}\right),\quad
\text{and}\quad
{\mathsf f}_{\rm f} \left(\begin{matrix}u\\ v\end{matrix}\right) =
\left(\begin{matrix} -c(u-v)\\ 0\end{matrix}\right).
\end{equation*}

\noindent
To compute numerical solutions of the linear system~\eqref{eq:cas-simple-lineaire},
we consider splitting schemes between the fast ({\it i.e.} first)
equation of the system and the slow ({\it i.e.} second) equation.
{Since the equilibrium points of the linear system
\eqref{eq:cas-simple-lineaire} are located on the line of equation
$u=v$, we require that the matrices $M_f(\lambda_f)$ and $M_s(\lambda_s)$ that
will constitute the fast and slow schemes 
are such that $M_f(\lambda_f)(1,1)^{\t}=(1,1)^{\t}$ and
$M_s(\lambda_s)(1,1)^{\t}=(1,1)^{\t}$ so that
these matrices preserve the asymptotics, and so do all their products.}
Therefore the  numerical schemes always lead to a product of matrices
of the form
\begin{equation}
\label{formedesmatrices}  
    M_f(\lambda_f):=\begin{pmatrix}
      \lambda_f  & 1-\lambda_f \\ 
       0  & 1
    \end{pmatrix}\ 
    \text{and }
   \  M_s(\lambda_s):=\begin{pmatrix}
      1  & 0 \\
       1-\lambda_s  & \lambda_s
    \end{pmatrix},
\end{equation}
where $s$ (resp. $f$) stands for ``slow'' (resp. ``fast'').
The parameters $\lambda_f$ and $\lambda_s$ are functions of the time-step $\dt$
with values in $(0,1)$
that depend on the choice of integrators (exact flow or $\theta$-scheme)
for the slow and fast equations.
The composition of the matrices
depends on the type of splitting one wants to use ({\it e.g. }Lie or Strang
type). For example, for a Lie-type splitting without subcycling where
the solutions to the fast and slow equations 
are approached by a forward Euler scheme of time-step $\dt$,
$\lambda_f=1-Nc\dt$ and $\lambda_s=1-c\dt$ and the matrix of the numerical scheme reads
\begin{equation*}
G(\dt)=M_s(\lambda_s)M_f(\lambda_f)=\begin{pmatrix}\lambda_f&1-\lambda_f\\\lambda_f(1-\lambda_s)&1-\lambda_f(1-\lambda_s)\end{pmatrix}.
\end{equation*}

\subsection{The exact solutions of the linear system \eqref{eq:cas-simple-lineaire}}
Let us compute the exact solution of \eqref{eq:cas-simple-lineaire}.
We consider the matrix
\begin{equation*}
  A=\begin{pmatrix} 
      -N  & N \\
       1  & - 1 
    \end{pmatrix}.
\end{equation*}
It is diagonalizable and its eigenvalues and associated spectral projectors are
\begin{equation*}
  \left(-(N+1),\,\,P_{\rm ex}=-A/(N+1)\right)\mbox{and }\left(0,\,\,Q_{\rm ex}=(1,1)^{\t}\,(1,N)/(N+1)\right).
 \end{equation*}
So the exact solution of system \eqref{eq:cas-simple-lineaire} is,
for all $t \in \R$,
\begin{equation*}
W(t):=   (u(t),v(t))^{\t}
    =\left({\rm e}^{-(N+1) ct}P_{\rm ex}+Q_{\rm ex}\right)
        (u^0,v^0)^{\t},
  \end{equation*}
\noindent for the initial values $u^0\in\R$ and $v^0\in\R$ at time $t=0$.
In particular, we note that all the solutions converge to the equilibrium state
$Q_{\rm ex}(u^0,v^0)^{\t}$ when $t$ tends to infinity.
In the following, we fix $T>0$ and define
\begin{equation}
F(T)={\rm e}^{cTA}={\rm e}^{-(N+1) cT}P_{\rm ex}+Q_{\rm ex},\label{flot_exact}
\end{equation}
the matrix of the exact flow at time $T$
of the system \eqref{eq:cas-simple-lineaire},
the eigenvalues of which are ${\rm e}^{-(N+1) cT}$ and $1$. 

\subsection{General properties of splitting schemes
for the linear system \eqref{eq:cas-simple-lineaire}}
\label{subsec:GenPropSplit}
Let $G(\dt)$ be defined for $\dt\in \mathcal{I}_N$ as the $2$-by-$2$ matrix
of any linear numerical flow that is a product of matrices of the form
\eqref{formedesmatrices}, where $\mathcal{I}_N$ is the intersection,
that may depend on $N$, of the stability intervals of the involved
schemes (see examples in Section \ref{LS}).
In the following, for all $n\in\N$, we will denote by
$$W^n:=(u^n,v^n)^{\t}=(G(\dt))^n W^0$$
the numerical solution at time $n\dt$ starting from the initial datum $W^0=(u^0,v^0)^{\t}$.
\begin{lemma}\label{boris}
For all $\dt\in \mathcal{I}_N$, the matrix $G(\dt) $ is diagonalizable, with two distinct real eigenvalues.
One of these eigenvalues is $1$ and the other one lies in $(0,1)$.
The vector $(1,1)^{\t}$ is an eigenvector of $G(\dt)$ associated to the
eigenvalue $1$. Hence the matrix $G(\dt)$ reads
\begin{equation}
\label{formeGdeltat}
  G(\dt) =
  \begin{pmatrix}
    1-\alpha(\dt) & \alpha (\dt) \\
    \beta(\dt) & 1 - \beta(\dt) \\
  \end{pmatrix},
\end{equation}
for two real-valued functions $\alpha$ and $\beta$. 
Moreover, the spectral decomposition of the matrix $G(\dt)$ reads
\begin{equation}
\label{spectralGdeltat}
  G(\dt) = \mu(\dt) P (\dt) + Q(\dt),
\end{equation}
where $P(\dt)$ is the matrix of the spectral projector of $G(\dt)$
associated
to the eigenvalue $\mu(\dt)=1-\alpha(\dt) - \beta(\dt)$
and $Q(\dt)$ is that associated to the eigenvalue 1. In particular,
\begin{equation}
\label{Qdeltat}
Q(\dt)=(1,1)^{\t}\,(\beta(\dt),\alpha(\dt))/(\alpha(\dt)+\beta(\dt)).
\end{equation}
\end{lemma}

\begin{proof}
  Since all the matrices $M_s$ and $M_f$ have $(1,1)^{\t}$ for eigenvector associated
with $1$, so does any (finite) product of such matrices and this explains the form of
the matrix $G(\dt)$ in \eqref{formeGdeltat}. Moreover, since all the matrices
$M_s$ and $M_f$  also have their other real eigenvalue in $(0,1)$, the determinant
of a product of such matrices is in $(0,1)$. Hence for all $\dt\in \mathcal{I}_N$, $G(\dt)$
is diagonalizable with eigenvalues $1$ and
$ \mu(\dt) = {\rm Tr}(G(\dt))-1 = {\rm det} (G(\dt))\in(0,1)$.
\end{proof}

\begin{remark}\label{Galphabeta}
  We will sometimes use in the following the notation $G[\alpha,\beta]$ in
  reference to \eqref{formeGdeltat}.
\end{remark}
\begin{remark}
  The functions $\alpha$ and $\beta$ are polynomials of functions of type
$\lambda_f$ and $\lambda_s$ which depend on the time integrators for the
split equations (see the form of the matrices $M_f$ and $M_s$ in
\eqref{formedesmatrices}; see also examples in page \pageref{tableau4schemas}).
\end{remark}

With Lemma \ref{boris}, we can show that the exact and numerical propagators
share an interesting property:
\begin{proposition}\label{prop:projection}
  For any fixed $\dt>0$, $(F(n\dt))$ projects the vector
  $(u^0,v^0)^{\t}$ onto the line of equation $u=v$ when $n$ tends to infinity and
  so does $(G(\dt))^n$ for all $\dt\in \mathcal{I}_N$.
\end{proposition}
\begin{proof}
Recall that for all $n\in\N$, $F(n\dt)=F(\dt)^n$.
The projection property for $F(n\dt)$ as $n\rightarrow+\infty$ relies on
the decomposition \eqref{flot_exact}. Using Lemma \ref{boris}, we get
for all $n\in\N$, $(G(\dt))^n = (\mu(\dt))^n P (\dt) + Q(\dt),$
with $|\mu(\dt)|<1$ and the result follows.
\end{proof}

Following the notations of Section \ref{sec:fw},
we denote the numerical and exact limits in time by
\begin{equation*}
  (u^\infty_{\rm num},v^\infty_{\rm num})^{\t} = \lim_{n\to +\infty} (G(\dt))^n (u^0,v^0)^{\t}
  \quad \text{and} \quad
  (u^\infty_{\rm ex},v^\infty_{\rm ex})^{\t} = \lim_{n\to +\infty} (F(\dt))^n (u^0,v^0)^{\t}.
\end{equation*}
Recall that the numerical limit $(u^\infty_{\rm num},v^\infty_{\rm num})^{\t}$
actually depends on $\dt$. In this context, we consider the asymptotic error
$\varepsilon^{\rm as} := (u^\infty_{\rm num},v^\infty_{\rm num})^{\t}
- (u^\infty_{\rm ex},v^\infty_{\rm ex})^{\t}$ and are interested in the asymptotic
order of the method $G(\dt)$ (see Section \ref{sec:conceptasympterror}).
Note that, for the linear system \eqref{eq:cas-simple-lineaire},
$\varepsilon^{\rm as}=(Q(\dt)-Q_{\rm
  ex})(u^0,v^0)^{\t}$. We define  $S(\dt)$ as the ratio $\alpha(\dt)/\beta(\dt)$.
Since for all $t\in\R$, $u(t)+Nv(t)=u(0)+Nv(0)$, and for all
$n\geq 0$, ${u^n+S(\dt)
  v^n = u^0 + S(\dt) v^0}$, $\varepsilon^{\rm as}$ can be measured in terms of the difference of the slopes of the two straight lines
$u+Nv=u^0+Nv^0$ and $u+S(\dt)v=u^0+S(\dt)v^0$
(see Figure \ref{fig:plan_uv}). 
More precisely,
\begin{equation}
\label{eq:erreurasympt}
  \|\varepsilon^{\rm as}\|_2 = \sqrt{2}
  \frac{N}{N+1}
  \frac{|u^0-v^0|}{(S(\dt)+1)}\,\frac{|S(\dt)-N|}{N}.
\end{equation}
Let us define the relative asymptotic error of the method $G(\dt)$ applied
to the linear problem \eqref{eq:cas-simple-lineaire}.
\begin{definition} 
  The relative asymptotic error is defined as the scaled difference
  \begin{equation*} 
    \varepsilon^\infty:=\dfrac{|S(\dt)-N|}{N}.
  \end{equation*}
\end{definition}
For the linear system~\eqref{eq:cas-simple-lineaire}, the asymptotic order is
studied in the following by means of the relative
asymptotic error $\varepsilon^\infty$.
As we shall see in the proof of Theorem \ref{th:ordreasymptotiquelineaire},
for consistency reasons, $S(\dt)\rightarrow N$ when $\dt\rightarrow 0$, so
in view of \eqref{eq:erreurasympt}, 
$\varepsilon^{\rm as}$ and $\varepsilon^\infty$ have the same order in $\dt$.
\\

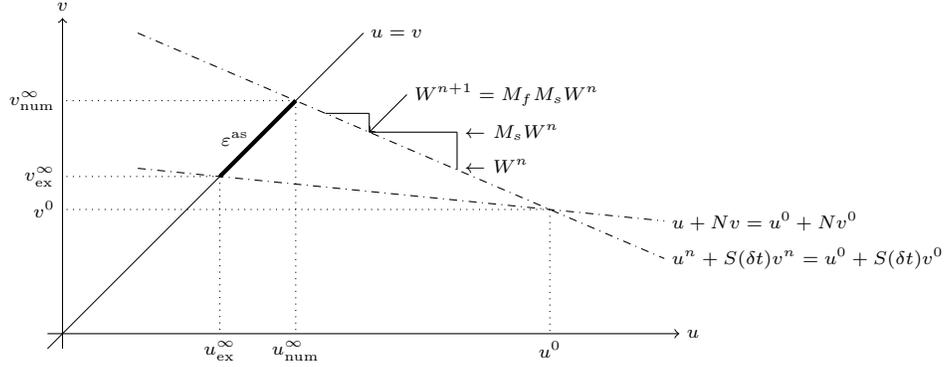
\begin{figure}
\centering
{\scriptsize
 \begin{tikzpicture}[domain=0:8]
     \path[name path=absc,draw,->] (-0.2,0) -- (8.2,0) node(uline)[right] {$u$};
     \path[name path=ord,draw,->] (0,-0.2) -- (0,4.2) node(vline)[above] {$v$};
     \path[name path=id,draw,ultra thin] (-0.2,-0.2) -- (4,4)         node[right] {$u=v$};
 \path[name path=exact,draw,dashdotted] (1,2.2) -- (8,1.5) node [right]{$u+Nv=u^0+Nv^0$} ;
  \path[name path=num,draw,dashdotted] (1,4) -- (8,1) node
  [right]{$u^n+S(\dt)v^n=u^0+S(\dt)v^0$} ; 
  \path [name intersections={of = exact and num}];
   \coordinate (I)  at (intersection-1);
 \draw[dotted] (I) -- (I |- uline) node[below]{$u^0$};
 \draw[dotted] (I) -- (I -| vline) node[left]{$v^0$};
  \path [name intersections={of = exact and id}]; 
   \coordinate (E)  at (intersection-1);
 \draw[dotted] (E) -- (E |- uline) node[below]{$u^\infty_{\rm ex}$};
 \draw[dotted] (E) -- (E -| vline) node[left]{$v^\infty_{\rm ex}$};
 \path [name intersections={of = num and id}];
   \coordinate (N)  at (intersection-1);
 \draw[dotted] (N) -- (N |- uline) node[below]{$u^\infty_{\rm num}$};
 \draw[dotted] (N) -- (N -| vline) node[left]{$v^\infty_{\rm num}$};
 \coordinate (D) at (3.5,0);
 \path[name path=Mn,draw=none]  (D) -- (7.5,5);
  \path [name intersections={of = num and Mn}];
   \coordinate (M)  at (intersection-1);
  \coordinate (M1) at ($(M)+(0,0.5)$);
 \draw (M) -- (M1) node[right]{$\leftarrow$ $M_sW^n$}; 
 \coordinate (Mplus) at ($(M)+(0,0.05)$);
 \node[right] at (Mplus) {$\leftarrow$ $W^{n}$};
 \path[name path=Mn1,draw=none] (M1) -- (M1 -| vline); 
 \path [name intersections={of = num and Mn1}];
   \coordinate (M2)  at (intersection-1);
 \draw (M1) -- (M2);
 \draw[<-] (M2) --  ($(M2)+(0.5,0.5)$)  node[above,right]{$W^{n+1}=M_fM_sW^n$};
  \coordinate (M3) at ($(M2)+(0,0.25)$);
 \draw (M2) -- (M3) node[right]{}; 
 \path[name path=Mn2,draw=none] (M3) -- (M3 -| vline); 
 \path [name intersections={of = num and Mn2}];
   \coordinate (M4)  at (intersection-1);
 \draw (M3) -- (M4);
 \path [name intersections={of = id and exact}];
   \coordinate (Ie)  at (intersection-1);
 \path [name intersections={of = id and num}];
   \coordinate (In)  at (intersection-1);
 \draw[ultra thick] (Ie) -- (In) node[pos=0.5,above,left]{$\varepsilon^{\rm as}$}; 
 \end{tikzpicture} 
}
\caption{Evolution of the exact and numerical solution in the phase space
  $\R_u\times\R_v$. We note $W^n=(u^n,v^n)^{\t}$.}
\label{fig:plan_uv} 
\end{figure}

Our first result is the following link between the final-time classical order
of a splitting method $G(\dt)$ defined as above for the solution
of System \eqref{eq:cas-simple-lineaire} and its A-order.
\begin{theorem}
\label{th:ordreasymptotiquelineaire}
Let $G(\dt)$ be defined for $\dt\in \mathcal{I}_N$, associated with a discretization of \eqref{eq:cas-simple-lineaire}
and assume that it
is a product of matrices of the form \eqref{formedesmatrices}.
If the local order of $G(\dt)$ is at least $p+1$
(so that its global order at least $p$),
then its A-order is at least $p$.
\end{theorem}

\begin{proof}
Since the numerical flow $G(\dt)$ has local order $p+1$,
its difference with the exact flow $F(\dt)$ reads
\begin{equation*}
  G(\dt) - F(\dt)=
  \begin{pmatrix}
    1-\alpha(\dt) & \alpha (\dt) \\
    \beta(\dt) & 1 - \beta(\dt) \\
  \end{pmatrix}
  -{\rm e}^{-(N+1) c \dt}P_{\rm ex}-Q_{\rm ex}
  = {\mathcal O}(\dt ^ {p+1}).
\end{equation*}
This implies the following Taylor expansions for $\alpha$ and $\beta$:
\begin{equation*}
  \alpha(\dt) = (1-{\rm e}^{-c(N+1)\dt})(N/(N+1)) 
  +{\mathcal O}(\dt ^ {p+1})
  \,\, \text{and} \,\,
  \beta(\dt) = (1-{\rm e}^{-c(N+1)\dt})/(N+1)
  +{\mathcal O}(\dt ^ {p+1}).
\end{equation*}
We infer that the slope of the equilibrium state is
$  S(\dt) = {\alpha(\dt)}/{\beta(\dt)} = N
  +{\mathcal O}(\dt ^ {p}).$
\end{proof}

\bigskip Now, we define splitting schemes for the linear differential
system \eqref{eq:cas-simple-lineaire}, based on the composition
of exact flows or $\theta$-schemes discretizing the split equations.
We focus on their asymptotic behavior.
We know from Proposition \ref{prop:projection} and
Theorem \ref{th:ordreasymptotiquelineaire} that
for all initial data $(u^0,v^0)\in\R^2$, the numerical solutions
provided by such splitting schemes (assuming they are consistent with
Equation (\ref{eq:cas-simple-lineaire})) converge to an asymptotic state
when the numerical time $n\dt$ tends to infinity (and $\dt$ is fixed).
The typical questions of interest are the following:
What is the size of this relative asymptotic error
with respect to the numerical time-step $\dt$ ?
Can we do better than the estimate on the A-order
provided by Theorem \ref{th:ordreasymptotiquelineaire} ?

\subsection{Lie, Strang, and weighted splitting schemes with and
without subcycling for the linear system \eqref{eq:cas-simple-lineaire}}
\label{LS}
Denoting by $\dt>0$ the numerical time-step related to the ``slow''
equation, the time-step associated to the ``fast'' equation is then $\dt/N$.
The (exact or numerical) integration of the fast (resp. slow) equation of \eqref{eq:cas-simple-lineaire} over a time-step
$\dt/N$ (resp $\dt$) yields the flow
\begin{equation*}
  \Phi_{f,\dt/N}\mbox{(resp. }\Phi_{s,\dt}\mbox{)} \qquad\mbox{with matrix }\qquad
  M_f(\lambda_f(\dt/N))\,\,\mbox{(resp. }M_s(\lambda_s(\dt))\mbox{)},
\end{equation*}
with
$\lambda_s(\dt),\lambda_f\left(\dt/N\right)\in\R$.
To fix the notations,
we write the Taylor expansions in $\dt$ of
$\lambda_s(\dt)$ and $\lambda_f(\delta t/N)$ in the following way:
\begin{equation}
  \label{eq:defAsAf}
   \lambda_f(\dt/N) = 1 - c \dt + c^2 A_f \dt ^2 + {\mathcal O}(\delta
t^3)
\quad
{\rm and}
\quad
\lambda_s(\dt) = 1 - c\dt + c^2 A_s \dt^2 + {\mathcal O}(\delta
t^3),
\end{equation}
where $A_f,A_s\in\R$ are the coefficients of interest.
For any functions $\lambda_f$, $\lambda_s$ of $\dt$, we consider the
following six schemes: given $i\in\{1,\dots,6\}$ and $W^n\in\R^2$, we set
$$W^{n+1}=G_i(\dt) W^n.$$
\begin{itemize}
\item {\bf Scheme \#1}: (Lie type - slow time - subcycled)
$$ G_1(\dt) = M_s(\lambda_s(\dt)) M_f(\lambda_f(\delta
t/N))^N $$
\item {\bf Scheme \#2}: (Lie type - fast time - no subcycling)
$$ G_2(\dt) = \big(M_s(\lambda_s(\dt/N)) M_f(\lambda_f(\delta
t/N))\big)^N$$
\item {\bf Scheme \#3}: (Strang type - slow time - subcycled)
$$ G_3(\dt) = M_s(\lambda_s(\dt/2)) \
M_f(\lambda_f(\dt/N))^N\ M_s(\lambda_s(\dt/2)) $$
\item {\bf Scheme \#4}: (Strang type - fast time - no subcycling)
$$ G_4(\dt) = \big(M_s(\lambda_s(\dt/(2N))) \
M_f(\lambda_f(\dt/N))\ M_s(\lambda_s(\dt/(2N))) \big)^N $$
\item {\bf Scheme \#5}: (weighted type (\cite{Csomos05}) - with subcycling)
$$ G_5(\dt) = \frac{1}{2} \left(
M_s(\lambda_s(\dt))
M_f(\lambda_f(\dt/N))^N
+
M_f(\lambda_f(\dt/N))^N
M_s(\lambda_s(\dt))
\right)
$$
\item {\bf Scheme \#6}: (weighted type - without subcycling)
$$ G_6(\dt) = \frac{1}{2^N} \big(
M_s(\lambda_s(\dt/N))
M_f(\lambda_f(\dt/N))
+
M_f(\lambda_f(\dt/N))
M_s(\lambda_s(\dt/N))
\big)^N
$$
\end{itemize}

\begin{remark}
\label{rem:Nunknown}
  Since in actual applications, the ratio $N$ between fast and slow scales
in the system may not be known accurately
(one may only know that it is, say, of order $10^3$),
the advantage of using subcyling techniques (with a subcycling number
of the same order as that of $N$) is that one can expect to achieve higher order
without having to know that ratio exactly, at least on the very
academic linear problem \eqref{eq:cas-simple-lineaire}.
\end{remark}

\begin{remark}
  When dealing with slow/fast Lie-splitting methods,
one has to choose which equation
will be integrated first: either the slow equation first, and then the fast one (which we denote by FS),
or the fast equation and then the slow one (which we denote by SF). {We chose this notation because of the usual
convention on the composition of flows:
the first to be applied is written on the right-hand side of the others.}
Note that, in our very simple
linear setting, the eigenvalues, eigenvectors, spectral projectors, etc,
of any FS splitting method can be deduced from those of a SF splitting formula
in a way explained in Appendix \ref{sec:FS_to_SF} and the analysis
extends straightforwardly.
Therefore, we restrict ourselves to the study of SF Lie-splitting schemes.
We also focus on FSF Stang-splitting schemes.
For weighted schemes, we take advantage of the symmetry and use
both SF and FS schemes.
\end{remark}

Using the notations of Lemma \ref{boris}, we obtain the results
presented in Table \ref{tableau4schemas}.
\begin{table}[!h]
  \begin{tabular}{|c|c|c|}\hline
    Scheme \#  & function $\alpha$ & function $\beta$ \\ [2mm] \hline 
    \#1 &$ \alpha_1(\dt) = 1-(\lambda_f(\dt/N))^N \qquad$ & $\qquad
    \beta_1(\dt) = (1-\lambda_s(\dt))
    (\lambda_f(\dt/N))^N$\\[2mm]\hline
    \#2 & 
    $ \alpha_2(\dt) = 1-\lambda_f(\dt/N) \qquad$ &$
    \qquad \beta_2(\dt) = (1-\lambda_s(\dt/N))
    \lambda_f(\dt/N)$\\[2mm]\hline
    \#3 &
    $\alpha_3(\dt) = (1-\lambda_f(\dt/N)^N)
    [\lambda_s(\dt/2)]^N\quad$ & $\quad \beta_3(\dt) = (1-\lambda_s(\delta
    t/2)) (1+[\lambda_f(\dt/N)^N]\lambda_s(\dt/2))$\\[2mm]\hline
    \#4 &
    $ \alpha_4(\dt) = (1-\lambda_f(\dt/N))
    \lambda_s(\dt/2) \quad$ &
    $\quad \beta_4(\dt) = (1-\lambda_s(\delta
    t/2)) (1+\lambda_f(\dt/N) \lambda_s(\dt/2))$\\[2mm]\hline
    \#5 &
    $ \alpha_5(\dt) = (1-\lambda_f(\dt/N))^N)
    (1+\lambda_s(\dt))/2 \quad$ &
    $\quad \beta_5(\dt) = (1-\lambda_s(\dt))
    (1+\lambda_f(\dt/N))^N)/2
    $\\[2mm]\hline
    \#6 &
    $ \alpha_6(\dt) = (1-\lambda_f(\dt/N)))
    (1+\lambda_s(\dt/N))/2^N \quad$ &
    $\quad \beta_6(\dt) = (1-\lambda_s(\dt/N))
    (1+\lambda_f(\dt/N)))/2^N
    $\\[2mm]
\hline
  \end{tabular}\bigskip
  \caption{The functions $\alpha$ and $\beta$ for the schemes \#1, \#2, \#3,
\#4, \#5, and \#6}
  \label{tableau4schemas}
\end{table}

\paragraph{Asymptotic order}
The above computations enable us to prove the following
\begin{proposition}[Lie splitting properties]
\label{prop:LiesansCsomos} 
Let $G(\dt)$ be a Lie splitting method such as
Schemes \#1 and \#2.\\
Then 
\begin{enumerate}
\item if $G(\dt)$ involves two methods of order at least 1,
then it has a classical order of at least
1 and an A-order of at least 1,
\item  if $G(\dt)$ involves two schemes of order at least 2,
then its A-order is at most 1,
\item however, there exists a combination of schemes
of order 1 such that $G(\dt)$ is a method of A-order at least 2
(even if its classical order is $1$).
\end{enumerate}
\end{proposition}
\begin{proof}
 \begin{enumerate}
 \item 
  The fact that Schemes \#1 and \#2 have a classical order of at least 1 follows
  from Theorem \ref{th:ordrenonlineaire}. The fact that their asymptotic
  order is at least 1 is granted by Theorem \ref{th:ordreasymptotiquelineaire}. 
  \item Let us consider Scheme \#1
  and write, using the Taylor expansions \eqref{eq:defAsAf},
\begin{equation}
\label{eq:TaylorS1}
S_1(\dt) = {\alpha_1(\dt)}/{\beta_1(\dt)} =
N+cN(A_s-A_f+(N+1)/2) \dt+ {\mathcal O}(\delta
t^2).
\end{equation}
When the two schemes are of order at least 2, we have $A_f=A_s=1/2$, so that the
A-order of $G_1(\dt)$ is exactly 1.
A similar computation yields
\begin{equation}
\label{eq:TaylorS2}
S_2(\dt) = {\alpha_2(\dt)}/{\beta_2(\dt)} =
N+c ((1-A_f) N+A_s) \dt + {\mathcal O}(\delta
t^2),
\end{equation}
so that the same conclusion is true for $G_2(\dt)$.
\item For $G_2(\dt)$, the choice $(A_f,A_s)=(1,0)$ leads to 
an A-order of at least 2 with two underlying methods of order 1 (see
the Taylor expansion \eqref{eq:TaylorS2} for Scheme \#2).
\end{enumerate}
\end{proof}

\begin{proposition}[weighted splitting properties]
\label{prop:CsomossansLie} 
Let $G(\dt)$ be a weighted splitting method such as
Schemes \#5 and \#6.\\
Then 
\begin{enumerate}
\item if $G(\dt)$ involves two methods of order 2,
then it has a classical order of at least 2 and an A-order of at least 2,
\item there exists a combination of schemes
of order 1 such that $G(\dt)$ is a  method of A-order at least 2
(even if its classical order is $1$),
\item moreover, using subcycling (Scheme \#5), there is a one-parameter
family (which does not depend on n) of couple of schemes of order 1
such that the corresponding weighted splitting is of A-order 2.
\end{enumerate}
\end{proposition}
\begin{proof}
 \begin{enumerate}
 \item 
  The fact that Schemes \#5 and \#6 have classical and asymptotic orders
at least 2 follows from Theorem \ref{th:ordrenonlineaire} and
Theorem \ref{th:ordreasymptotiquelineaire}. 
  \item Let us compute
\begin{equation}
  \label{eq:TaylorS5}
  S_5(\dt)= N+N c \left( A_s - A_f  \right) \dt
  + {\mathcal O}(\dt^2),
\end{equation}
and
\begin{equation}
  \label{eq:TaylorS6}
  S_6(\dt)=
N+c \left( N-1-2 N A_f+2A_s \right)
\dt/2
  + {\mathcal O}(\dt^2).
\end{equation}
Letting $A_f=A_s\neq 1/2$ in \eqref{eq:TaylorS5}
proves the result for couples of methods of order 1.
\item In order to obtain a method of A-order 2,
one needs to solve $A_s=A_f$ for Scheme \#5 and
$(N A_f-A_s)=(N-1)/2$ for Scheme \#6.
This proves the result.
\end{enumerate}
\end{proof}

\begin{remark}
The fact that a combination of two methods of classical order 1
can lead to a method of asymptotic order 2 is highly remarkable
since such a combination is in general of asymptotic order 1
as one can check on the Taylor expansions above.
However, in general, the coefficients $A_f$ and $A_s$ defining the methods
involved in a such combination that achieves asymptotic order 2 depend
on the ratio parameter $N$. We stress here that, in some
cases, using subcyling on appropriate methods, one can choose the coefficients
to be independent of $N$ (Scheme \#5). Note that, without subcycling,
in general, the coefficients $A_s$ and $A_f$ required to reach order 2
with weighted splitting Scheme \#6 do depend on $N$(except when the two
underlying methods are themselves of order 2 when $A_s=A_f=1/2$).
One can see in such a feature an advantage of using methods involving
subcyling.
\end{remark}

Let us now describe some properties of Strang splitting schemes.
\begin{proposition}[Strang splitting properties]\label{prop:Strang}
Let $G(\dt)$ be a Strang-splitting method such as Schemes \#3 and \#4.\\
Then
\begin{enumerate}
\item if $G(\dt)$ involves schemes of order at least 2, $G(\dt)$
has an order of at least 2 and an A-order of at least 2,
\item if $G(\dt)$ involves a scheme of order 1, $G(\dt)$ is of order 1,
but there exists a one-parameter family of 
schemes of order 1 such that the A-order of $G(\dt)$ is 2.
\end{enumerate}
\end{proposition}

\begin{proof} 
 \begin{enumerate}
\item  The fact that a Strang-splitting method involving two methods of
order 2 is of order at least 2 comes from Theorem \ref{th:ordrenonlineaire}.
The fact that its A-order is at least 2 follows from Theorem
\ref{th:ordreasymptotiquelineaire}.
\item  Assume we have the same Taylor expansion of $\lambda_f$ and $\lambda_s$
as in the proof of Proposition \ref{prop:LiesansCsomos} and Proposition \ref{prop:CsomossansLie}.
  For Scheme \#3, we have
\begin{equation}
\label{eq:TaylorS3}
S_3(\dt) = {\alpha_3(\dt)}/{\beta_3(\dt)}=
N+ N c (2 A_s-1+2-4 A_f) \dt / 4 + {\mathcal O}(\dt^2),
\end{equation}
and for Scheme \#4
\begin{equation}
\label{eq:TaylorS4}
S_4(\dt) = {\alpha_4(\dt)}/{\beta_4(\dt)}=
N+c (N(2 A_f -1)+2-4 A_s) \dt /4 + {\mathcal O}(\dt^2).
\end{equation}
One infers the equations to solve for $A_f$ and $A_s$ to prove the result.
For example, one can choose $(A_f,A_s)=(1/4,0)$ to have a Scheme \#3 of 
A-order at least 2 involving two schemes of order 1.\\
  \end{enumerate}
  \end{proof}
  
  \begin{remark}
    \label{remarqueordreplus}
  In contrast to what occurs in the Lie case, the dependence upon $N$
  in the Strang subcycled scheme \#3 is decoupled from the combination of $A_f$
  and $A_s$. 
  \end{remark}
\begin{remark}
\label{remarkFSF}
  We can exchange the influence of the choices of $A_s$ and $A_f$ in the A-order
  by Strang-splitting with the order FSF, that is, by introducing
  \begin{align*}
    \widetilde{G_3}(\dt)& = M_f(\lambda_f(\dt/(2N)))^N \ M_s(\lambda_s(\dt))\
    M_f(\lambda_f(\dt/(2N)))^N,\\
 \widetilde{G_4}(\dt)& = (M_f(\lambda_f(\dt/(2N))) \ M_s(\lambda_s(\dt))\
    M_f(\lambda_f(\dt/(2N))))^N,
  \end{align*}
  thanks to the computations detailed in Appendix \ref{sec:FS_to_SF}.
  The coefficient in front of $\dt^2$ is then $1-4A_s+2A_f$
  (resp. $2(2A_s-1)+N(1-2A_f)$) for Scheme $\widetilde{\#3}$ (resp.
  $\widetilde{\#4}$).
\end{remark}

\paragraph{Convergence rate}

Let us perform the same analysis on the convergence rate to equilibrium, {\em
  i.e.} the eigenvalues $\mu_i(\dt)$, $i\in\{1,\hdots,4\}$
of the matrices $G_i(\dt)$ defined in Lemma \ref{boris}.
We get the Taylor expansions of
\begin{equation*}
  \rho_i(\dt)=\mu_i(\dt)-{\rm
  e}^{-c(N+1)\dt},
\end{equation*}
that we summarize in Table \ref{tableau2rho}.
\begin{table}[!h]
  \begin{center}
  \begin{tabular}{|c|c|}
    \hline\hspace*{1mm}
    $i$& $(A_f,A_s)$\\[1mm]\hline
    $\rho_1(\dt)$&$c^2(N(2A_f-1)+2A_s-1)\dt^2/2+\mathcal{O}(\dt^3)$  \\[1mm]\hline
     $\rho_2(\dt)$&$c^2(N^2(2A_f-1)+2A_s-1)\dt^2/(2N)+\mathcal{O}(\dt^3)$  \\[1mm]\hline
   $\rho_3(\dt)$&$c^2(2N(2A_f-1)+2A_s-1)\dt^2/4+\mathcal{O}(\dt^3)$  \\[1mm]\hline
    $\rho_4(\dt)$&$c^2(2N^2(2A_f-1)+2A_s-1)\dt^2/(4N)+\mathcal{O}(\dt^3)$  \\[1mm]\hline
  $\rho_5(\dt)$&$c^2(N A_f + A_s - (N+1)/2)\dt^2+\mathcal{O}(\dt^3)$  \\[1mm]\hline
$\rho_6(\dt)$&$c^2\left( N^2(2 A_f-1)+(2 A_s-1)\right) \dt^{2}/(2N)
+\mathcal{O}(\dt^3)$  \\[1mm]\hline
  \end{tabular}
  \caption{The functions $\rho$ for the schemes \#1, \#2, \#3, \#4, \#5 and \#6}
  \label{tableau2rho}
  \end{center}
\end{table}
One notes at once that second order fast and slow schemes generate a second
order approximation of the convergence rate, as well as an A-order of 2 for
Schemes \#3 \#4, \#5 and \#6.
Besides, one can manage to construct a second order
approximated rate by choosing at least one of the fast and slow schemes to be of
order 1, but the A-order will then be exactly 1.

\paragraph{Application to $\theta$-schemes}
In this paragraph, we consider two
$\theta$-schemes for the numerical solutions of the fast and slow equations
of system \eqref{eq:cas-simple-lineaire}.
We take $(\theta_f,\theta_s)\in[0,1]^2$ and we set
\begin{equation*}
    \lambda_{f}(\dt)=  \dfrac{1-Nc(1-\theta_f)\delta
      t}{1+\theta_f Nc\dt}
    \qquad \text{and} \qquad 
    \lambda_{s}(\dt)=  \dfrac{1-c(1-\theta_s)\delta
     t}{1+\theta_s c\dt}.
\end{equation*}
 In particular, we have
\begin{equation}
\label{eq:LienAtheta}
  (A_{f},A_{s})=(\theta_f,\theta_s).
\end{equation}

Classically, in order to ensure that the associated schemes  are 
A-stable in the classical sense (see \cite{hw}),
in case $\theta_f\in[0,1/2)$ (resp. $\theta_s\in[0,1/2)$), we assume that
$(1-2\theta_f)cN\dt/N <2$ (resp.   $(1-2\theta_s)c\dt <2$) so that
$\lambda_f^{\theta_f}(\dt/N)\in(0,1)$
(resp. $\lambda_s^{\theta_s}(\dt)\in(0,1)$).
The stability interval $\mathcal{I}_N$ defined at the beginning of Section
\ref{subsec:GenPropSplit} is the intersection of the corresponding domains in
$\dt$. Our choice of different time-steps for the slow and fast equations
in order to use subcycling techniques implies that $\mathcal{I}_N$
is independent of $N$ in that case.

\medskip

The results of the previous paragraphs provide us with the following
propositions, when the underlying numerical integration methods are
$\theta$-schemes. For Lie-splitting methods (Schemes \#1 and \#2)
and the weighted splitting methods (Schemes \#5 and \#6):

{
\begin{proposition}[Lie and weighted splitting methods involving
$\theta$-schemes]
  \label{Prop:LieCsomosTheta1}
  Assume $N>1$.
  \begin{enumerate}
  \item  Lie-splitting with $\theta$-schemes :
  The only scheme of type \#1 or \#2 of A-order at least 2 involving two
  $\theta$-schemes is of type \#2 with $\theta_s=0$ (fully explicit)
  and $\theta_f=1$ (fully implicit).
  In this very particular case, the A-order is infinite
  because $\alpha_2=N\beta_2$.
  \item weighted splitting methods with $\theta$-schemes :
There exists two one-parameter families of schemes involving
$\theta$-schemes, one of type \#5 and another one of type \#6, with
A-order 2.
  \end{enumerate}
\end{proposition}
}
\begin{proof}
\begin{enumerate}
  \item Plugging relation \eqref{eq:LienAtheta} in the Taylor expansions
\eqref{eq:TaylorS1} and \eqref{eq:TaylorS2}, the result follows
by cancelling the terms of order 1.
  \item Plugging \eqref{eq:LienAtheta} into the Taylor expansions
\eqref{eq:TaylorS5} and \eqref{eq:TaylorS6} yields the result.
\end{enumerate}
\end{proof}

\begin{remark}
  Note that, if a fully implicit scheme is at hand
  for the fast equation, it seems unwise to use a subcycling technique anyway,
  since there is no stability constraint on $\dt$ from the fast scheme part.
\end{remark}

\begin{remark}
\begin{enumerate}
  \item Concerning the Lie-splitting methods involving $\theta$-schemes,
one can check that no Lie-splitting scheme of type \#1 or \#2 has A-order
2 with an approximation of order 3 of the rate of convergence
(see relation \eqref{eq:LienAtheta} and
the Taylor expansions in the first two lines of Table \ref{tableau2rho}).
  \item Concerning the weighted splitting methods involving $\theta$-schemes,
note once again that the one-parameter family is independent of $N$
for the subcycled weighted scheme \#5, while it depends on $N$
for the non-subcycled weighted scheme \#6. This is an extra advantage
of subcycled schemes when $N$ is not known exactly
(see Remark \ref{rem:Nunknown}). Moreover,
using a weighted scheme with subcycling (type \#5)
allows to use a composition of two explicit schemes ($\theta_f=\theta_s=0$)
which has A-order 2 (see Fig. \ref{fig:Aorder2}).
\end{enumerate}
\end{remark}
\begin{figure}\centering
 \begin{tikzpicture}[scale=5]

 \draw[->] (0,0) -- (1.1,0) node[right] {$\theta_s$};

 \draw[->] (0,0) -- (0,1.1) node[above] {$\theta_f$};

 \draw[thick,dashdotted] (0,0.25) -- (1,.75) node[right] {$\# 3$ (slope $0.5$)};

 \draw[thick,dashdotted] (0,0.45) -- (1,.55) node[right] {$\# 4$ (slope $2/N$)};

 \draw[thick] (0,0) -- (1,1) node[right] {$\# 5$ (slope $1$)};

 \draw[thick,dashed] (0.25,0) -- (.75,1) node[above] {$\widetilde{\# 3}$ (slope $2$)};

 \draw[dotted] (0.5,0) -- (0.5,0.5);

 \draw[dotted] (0,0.5) -- (0.5,0.5);

 \draw[dotted] (0.75,0) -- (0.75,1);

 \draw[thick] (0,1) -- (1,1);

 \draw[thick] (1,0) -- (1,1);

 \draw[dotted] (0,0.75) -- (1,0.75);

 \node at (0,0.5) [left] {$0.5$};

 \node at (0.5,0) [below] {$0.5$};

 \node at (0,0.25) [left] {$0.25$};

 \node at (0.25,0) [below] {$0.25$};

 \node at (0,0.75) [left] {$0.75$};

 \node at (0.75,0) [below] {$0.75$};

 \node at (0,0) [left] {$0$ (explicit)};

 \node at (1,0) [below] {$1$};

 \node at (0,1) [left] {$1$ (implicit)};

 \end{tikzpicture}
\caption{One-parameter families of splitting schemes of A-order 2}
\label{fig:Aorder2}
\end{figure}
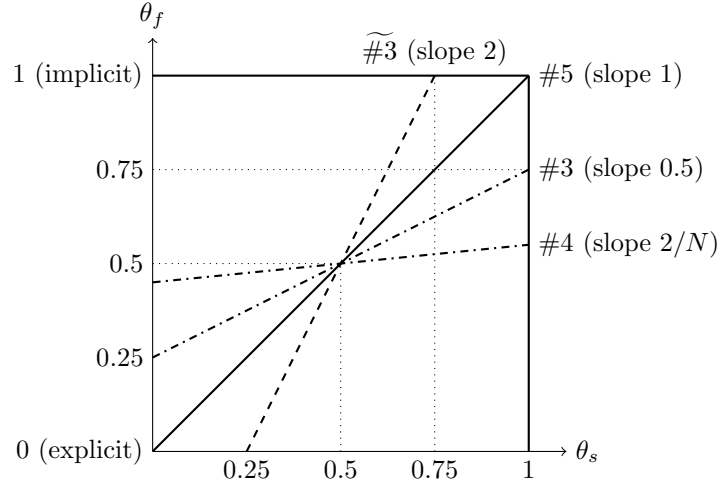


{
\begin{proposition}[Strang-splitting methods involving $\theta$-schemes]
  \label{prop:StrangTheta1}
Assume $N>1$.
\begin{enumerate}
\item There exists a one-parameter family of schemes of type \#3 with A-order
2, and another one of schemes of type \#4 with A-order 2.
\item  Using $\theta$-schemes, it is then possible to build a scheme of type
$\widetilde{\#3}$ (see Remark \ref{remarkFSF}) of A-order 2
with an explicit fast scheme ($\theta_f=0$) and
a semi-implicit slow scheme ($\theta_s=1/4$).
\item  Using the Strang-splitting (Schemes \#3 and \#4), the only combination of
$\theta$-schemes leading to a third order approximated rate of convergence
and having A-order 2 consists in taking the Crank-Nicolson
scheme for both the fast and slow schemes.
\end{enumerate}
\end{proposition}
}

\begin{proof}
\begin{enumerate}
\item  Plugging relation \eqref{eq:LienAtheta} in the Taylor expansions
\eqref{eq:TaylorS3} and \eqref{eq:TaylorS4}, the result follows
by cancelling the terms of order 1.
\item The coefficient in front of $\dt^2$ in the asymptotic
error expansion is then $4\theta_s-2+1-2\theta_f=0$.
\item  For schemes of type \#3,
plugging the relation \eqref{eq:LienAtheta} in the Taylor expansion
\eqref{eq:TaylorS3} and cancelling the term of order 1 yields a link
between $\theta_f$ and $\theta_s$ which does not match the condition
of cancellation of the term of order 2 in $\rho_3(\dt)$
(see Table \ref{tableau2rho}) except when $(\theta_f,\theta_s)=(0.5,0.5)$.
The proof is very similar for schemes of type \#4.
\end{enumerate}
\end{proof}

\begin{remark}
\begin{enumerate}
\item Concerning the Strang-splitting methods involving $\theta$-schemes,
without subcycling (Scheme \#4), the one-parameter family of
schemes depends on $N$ through the equation $2N(1-2\theta_s)+2\theta_f-1=0$.
On the contrary, with subcycling (Scheme \#3), 
the one-parameter family is independant of $N$ (since the link
between $\theta_f$ and $\theta_s$ is $4\theta_f-2+1-2\theta_s=0$).
\item Once again, in addition to having more reasonable computational costs
and relaxing stability constraints, using subcycling techniques allows
to derive families of schemes involving explicit schemes and with
reasonable high A-order (2, in this example with a Strang composition method).
\end{enumerate}
\end{remark}





  



 


\subsection{Conclusion}
Let us remind the reader that the applications we have in mind
are by far more complicated than the system \eqref{eq:cas-simple-lineaire}.
However, they share with the system \eqref{eq:cas-simple-lineaire}
the property that they involve a fast equation for which an implicit scheme
is costly or hard to solve, thus implying the use of
an explicit scheme, inducing a stability constraint on the numerical time-step
$\dt$. In that case, the subcycling techniques are computationally less costly,
thus relevant.

We proved in this section that, in view of the aforementioned goal, we can
indeed build two schemes, one of type $\widetilde{\#3}$ (Strang with subcycling)
with $\theta_f=0$ (explicit) and $\theta_s=1/4$ (semi-implicit),
and one of type \#5 (weighted) with $\theta_s=\theta_f=0$ (explicit/explicit)
which are of A-order 2, even though they
are (locally) consistent of order 1 with \eqref{eq:cas-simple-lineaire}
and have a rate of convergence which approximates the exact rate at order 2.
Moreover, the coefficients $\theta_f$ and $\theta_s$ of these schemes
are independent of $N$.

We postpone the numerical illustration of these results to the study of a
nonlinear system in the following section.

\section{Numerical tests of the asymptotic error of splitting schemes applied to a nonlinear toy-model}
\label{sec:nonlinearanalysis}
In Section \ref{sec:nonlinearanalysis}, the second system that is analyzed
is nonlinear and reads
\begin{equation}
  \label{eq:cas-simple}
  \begin{cases}
    u'=-Nc(u-v)-N(u-v)^2\\
    v'=c(u-v)+(u-v)^2.
  \end{cases}
\end{equation} 
With the notations of Section \ref{sec:Intro}, this means
\begin{equation*}
X=\R^2,\quad W=\left(\begin{matrix}u\\ v\end{matrix}\right),\quad
{\mathsf f}_{\rm s} \left(\begin{matrix}u\\ v\end{matrix}\right) =
\left(\begin{matrix}0\\ c(u-v)+(u-v)^2\end{matrix}\right),\quad
\text{and}\quad
{\mathsf f}_{\rm f} \left(\begin{matrix}u\\ v\end{matrix}\right) =
\left(\begin{matrix} -c(u-v)-(u-v)^2\\ 0\end{matrix}\right).
\end{equation*}

\subsection{Analysis of the exact solutions of the system \eqref{eq:cas-simple}}

In this section, we investigate the longtime behavior of the two-scale
nonlinear system \eqref{eq:cas-simple}.
Let us first write this system in the form
\begin{equation}
\label{eq:cas-simple2}
  \left\{
  \begin{matrix}
    u' & = & -N(u-v) [c+(u-v)] \\
    v' & = & (u-v ) [c+(u-v)].\\
  \end{matrix}
  \right. 
\end{equation}

This way, we are able to derive the following
\begin{proposition}
\label{prop:AnalyseSystemeNL}
Let $(u^0,v^0)\in\R^2$ be given. The maximal solution starting at
$(u^0,v^0)$ lies on the straight line of equation $u+Nv=u^0+Nv^0$.
It is defined for all non-negative time if $u^0+c\geq v^0$
and it {ceases to exist} after a finite positive time
if $u^0+c< v^0$.
Moreover, if $u^0+c=v^0$ then the solution is constant, and if $u^0+c>v^0$
then the solution tends to the intersection of the two straight lines
of equations $u+Nv=u^0+Nv^0$ and $u=v$, {\it i.e. } to the point
of coordinates $(u^0+Nv^0)/(N+1)\times(1,1)$.
\end{proposition}

\begin{proof}
The linear change of variable $(X,Y)=(u+Nv,u-v)$ yields the
equivalent differential system
\begin{equation*}
\begin{cases}
    X'  =  0 ,\\
    Y'  =  -(N+1) Y (c+Y).
\end{cases}
\end{equation*}
The second equation of this system has for maximal solution starting at
$t=0$ in $Y^0\in\R$ the function
${Y(t) = Y^0 {\rm e}^{-c(N+1)t}/(1+(1-{\rm e}^{-c(N+1)t})Y^0/c)}$
defined as long as $-c<Y^0(1-{\rm e}^{-c(N+1)t})$.
The result on the existence time for the maximal solutions of \eqref{eq:cas-simple2} follows from this observation. Moreover, if $Y^0>0$, then
$Y(t)$ tends to $0$ when $t$ tends to $+\infty$. This proves the asymptotic
behavior of the corresponding maximal solutions.
\end{proof}

Hence, for the range of interest of initial values ($(u^0,v^0)$ such that
$u^0+c > v^0$), the qualitative behavior is the same for the linear
system \eqref{eq:cas-simple-lineaire} and for the nonlinear system
\eqref{eq:cas-simple2}: the solutions evolve on straight lines
of equation $u+Nv={C}${, where $C$ is a constant,} and converge
to an equilibrium point located on the line of equation $u=v$.
Therefore, we extend the Definition \ref{def:monique}
of the asymptotic error $\varepsilon^{\rm as}$ to 
this nonlinear case as well.

\subsection{Splitting schemes with or without subcycling for
the nonlinear problem \eqref{eq:cas-simple}}

In the following, we consider numerical splitting methods for the
nonlinear problem \eqref{eq:cas-simple} in the same way as for
the linear problem \eqref{eq:cas-simple-lineaire} in Section \ref{LS}:
\begin{itemize}
\item Scheme \#1 is a SF Lie-splitting method with subcycling,
\item Scheme \#2 is a SF Lie-splitting method without subcycling,
\item Scheme \#3 is a FSF Strang-splitting method with subcycling,
\item Scheme \#4 is a FSF Strang-splitting method without subcycling,
\item Scheme \#5 is a weighted splitting method with subcycling, and
\item Scheme \#6 is a weighted splitting method without subcycling.
\end{itemize}
Once again, we use $\theta$-schemes to integrate the split equations
numerically:
we chose $(\theta_f,\theta_s)\in[0,1]^2$ and define $\Phi_{f,\dt}$ and
$\Phi_{s,\dt}$ as follows. For the fast equation,
the first component $u^{n+1}$ of $\Phi_{f,\dt}(u^n,v^n)$ solves the equation in $X$
\begin{equation*}
  X-u^n=N\dt(1-\theta_f)\big(c(v^n-u^n)-(u^n-v^n)^2\big)
  +N\dt\theta_f\big(c(v^n-X)-(X-v^n)^2\big),
\end{equation*}
while its second one is its second argument $v^n$. For the slow equation,
the second component $v^{n+1}$ of $\Phi_{s,\dt}(u^{n+1},v^n)$ solves 
the equation in $X$
\begin{equation*}
  X-v^n=\dt(1-\theta_s)\big(c(u^{n+1}-v^n)+(u^{n+1}-v^n)^2\big)
  +\dt\theta_s\big(c(u^{n+1}-X)+(u^{n+1}-X)^2\big),
\end{equation*}
while its first component is its first argument $u^{n+1}$.

\subsection{Numerical examples of splitting methods for problem \eqref{eq:cas-simple}}

We run the six schemes with six different values of the couple
$(\theta_f,\theta_s)$. We sum up the results on the asymptotic order
in Table \ref{tableauNL} and  provide numerical results in Figure
\ref{fig:nonlinear01}.
These results were obtained with final time $T=5.0$, speed $c=1$,
factor $N=10$, initial datum $(u^0,v^0)=(5,1)$, so that, using
the analysis carried out in the proof of Proposition
\ref{prop:AnalyseSystemeNL},
the exact solution at final time is within a distance smaller than $10^{-20}$
of its asymptotic limit $15/11\times(1,1)^{\t}$.
We then used as an approximation of the asymptotic error the
difference between $W_{\rm ex}(5.0)$ and $W_{\rm num}^{n_{\rm end}}$
where $n_{\rm end}$ is such that $\dt \cdot n_{\rm end}=5.0$.

By Theorem \ref{th:ordrenonlineaire},
we know that the Lie-splitting schemes (Scheme \#1 and Scheme \#2)
are of classical order $1$ for any possible choice of $(\theta_f,\theta_s)$.
The first two columns of Table \ref{tableauNL} show that the asymptotic order
is also $1$ in these cases, except when $(\theta_f,\theta_s)=(1,0)$.
This is in accordance with the results obtained in
Proposition~\ref{Prop:LieCsomosTheta1} for the linear system
\eqref{eq:cas-simple-lineaire} since in this case, the A-order
of Schemes \#1 and \#2 is 1 except when $(\theta_f,\theta_s)=(1,0)$ and
the A-order is infinite (see Proposition \ref{Prop:LieCsomosTheta1}).
Theorem \ref{th:ordrenonlineaire} also implies that
the Strang-splitting scheme \#3 is at least of classical order $1$ with the choice
$(\theta_f,\theta_s)=(0,1)$ and the asymptotic orders collected in the middle of the third line of Table \ref{tableauNL} show that the numerical asymptotic order
is also $1$ in this case.
The same theorem also ensures that Scheme \#3 has order $2$ when applied
with $(\theta_f,\theta_s)=(1/2,1/2)$. The asymptotic orders displayed
in the middle of the fourth line of Table \ref{tableauNL} show that the
asymptotic order is also $2$ in this case. 
The last two lines are even more interesting:
for $(\theta_f,\theta_s)=(0,1/4)$ and $(\theta_f,\theta_s)=((N+1)/(2N),3/4)$,
the classical order of the Strang splitting method is, by Theorem
\ref{th:ordrenonlineaire} at least $1$. In the first case
$(\theta_f,\theta_s)=(0,1/4)$, the numerical results suggest that the subcycled
Scheme \#3 has A-order $2$ while the non-subcycled Scheme \#4 has
A-order $1$. We recall that, for these parameters, the Scheme \#3
was of A-order $2$ in the linear setting (see Remark \ref{remarqueordreplus}).
In the second case $(\theta_f,\theta_s)=((N+1)/(2N),3/4)$, the
same phenomenon occurs: Scheme \#3 has A-order $1$ while Scheme \#4 has
A-order $2$. We recall that these values of the parameters were
chosen in the linear setting in such a way that the Scheme \#4
has A-order $2$.
The weighted splitting scheme without subcycling (Scheme \#6) applied to
the nonlinear problem \eqref{eq:cas-simple2} is of numerical A-order 1
except when $\theta_s=\theta_f=1/2$ and the numerical A-order is 3
(see Table \ref{tableauNL}). This is in good accordance with results
for the linear case proved in Section \ref{LS} since, for the linear
problem \eqref{eq:cas-simple-lineaire}, we have
\begin{equation*}
  S_6(\dt) = N+\frac{1}{2}\,c \left( 2\theta_s-1-N(2\theta_f-1) \right)\dt
+\frac{1}{4}\,c^{2} \left( (1-2\theta_f)(1-N+2N\theta_f-2\theta_s)\right) \dt^2
+{\mathcal O}(\dt^3),
\end{equation*}
and the terms of order 1 and 2 in the Taylor expansion of $S_6(\dt)$
vanish for these values of $\theta_s$ and $\theta_f$.
The weighted splitting scheme with subcycling (Scheme \#5) applied
to the nonlinear problem \eqref{eq:cas-simple2} is
indeed of numerical A-order 2 in general when $\theta_f=\theta_s$,
and is of numerical A-order 1 in other cases.
The two relatively high values on the last 2 lines of the corresponding
row of Table \ref{tableauNL} are due to the fact that $\dt$ was not
small enough to reach the actual rate.
These results are in good accordance with the results proved for the linear
problem \eqref{eq:cas-simple-lineaire}
(see \eqref{eq:TaylorS5} and \eqref{eq:LienAtheta}).

Roughly speaking, a subcycled scheme (odd number) requires half as many
numerical computations as the corresponding not-subcycled scheme
(even number), since the computational ratio is of order $(N+1)/(2N)\sim 1/2$.
Therefore, for a given precision $\varepsilon>0$ to be achieved on the
asymptotic state,
the previous analysis suggests to use a subcycled method with high order.
For example, for the integration of the nonlinear problem
\eqref{eq:cas-simple2}, provided $T>0$ is chosen big enough,
the subcycled Scheme \#3, which has A-order 2 (and whose
coefficients $\theta_f$ and $\theta_s$ do not depend on the value of
$N$ (see Remark \ref{rem:Nunknown}),
will require ${\mathcal O}((N+1)\times T/\varepsilon^{1/2})$ computations,
while its not-subcycled analogue Scheme \#4, which has A-order 1, will require
${\mathcal O}(2N\times T/\varepsilon)$ computations.

\subsection{Conclusion} 

These examples suggest that, in this context, the A-order of a scheme
applied to the linear problem is the same as the A-order of the scheme
applied to the nonlinear problem.
This can be explained by the fact that the two problems
\eqref{eq:cas-simple-lineaire} and \eqref{eq:cas-simple} have the same
set of attractive equilibrium points (the straight line $u=v$),
they project the initial datum $(u^0,v^0)$
(chosen in an appropriate subset of the phase plane ($u^0+c>v^0$)) on the same
equilibrium point $(u^0+Nv^0)/(N+1)\times(1,1)$, and in the neighborhood
of this equilibrium point, $(u-v)^2 << |u-v|$.
In particular, these examples show that it is possible
to build in the nonlinear setting, as well in the linear setting,
splitting methods with asymptotic order greater than the classical order
of the schemes used for solving the split-equations. {We expect
that the A-order is the same for the linear and nonlinear problems,
at least for problems admitting a ``sufficiently attractive'' stationary state,
perhaps in terms of existence of a Lyapunov functional. We proposed a theoretical
framework in Section \ref{sec:fw}.
However, finding a theoretical framework which is not too abstract,
allows for rigorous proof 
(of the asymptotic order of the splitting methods with and without subcycling),
and includes sufficiently many interesting applications (and in particular PDEs
examples as in the next section) seems out of reach for the authors right now.}

{\begin{table}[!h]
  \begin{center}
    \begin{tabular}{|c|c|c|c|c|c|c|}
      \hline
      $(\theta_f,\theta_s)$   & Scheme \#1  &Scheme \#2 & Scheme \#3 & Scheme \#4 & Scheme \#5 & Scheme \#6 \\
      \hline
      $(1.0,0.0)$            & $0.8642$ & -- & $0.7700$ & $0.9787$ & $1.2941$ & $1.0001$\\
      $(0.0,0.0)$             & $0.8693$ & $1.0072$ & $1.4769$ & $1.0409$  & \underline{$1.9647$} & $1.0055$  \\
      $(0.0,1.0)$             & $0.8404$ & $1.0000$ & $1.1860$ & $1.0229$ & $0.8734$ & $1.0002$\\
      $(0.5,0.5)$             & $0.8534$ & $1.0000$ & \underline{$1.8313$} & \underline{$1.9984$} & \underline{$1.8888$} &\underline{$2.4817$}\\
      $(0.0,0.25)$            & $0.8617$ & $1.0053$ & \underline{$1.8674$} & $1.0354$ & $1.8373$ & $1.0042$\\
      $(\frac{N+1}{2N},0.75)$ & $0.8463$ & $0.9975$ & $1.3994$ & \underline{$1.9926$} & $1.8184$ & $0.9955$\\
      \hline
    \end{tabular}
  \end{center}
\caption{Asymptotic error for the 6 schemes for some values of
($\theta_f,\theta_s)$.
Figures are underlined when the method is of A-order at least 2.} 
\label{tableauNL}   
\end{table}}
\begin{figure}[!h]  
  \centering   
  \subfigure[]{
 \begin{tikzpicture}

 \begin{axis}[%
 width=3.82222222222222cm,
 height=2.80333333333333cm,
 scale only axis,
 xmin=-3.2,
 xmax=-1.2,
 ymin=-14,
 ymax=0,
 ylabel={$\log_{10}(\varepsilon_{\rm as})$},
 axis x line*=bottom,
 axis y line*=left
 ]
 \addplot [color=red,solid,forget plot]
   table[row sep=crcr]{-1.30102999566398	-0.552382201974175\\
 -1.60205999132796	-0.72886568052388\\
 -1.90308998699194	-0.955928246070231\\
 -2.20411998265592	-1.2185452105408\\
 -2.50514997831991	-1.50088757650633\\
 -2.80617997398389	-1.79290562548508\\
 -3.10720996964787	-2.08954304449394\\
 };
 \addplot [color=red,dashed,forget plot]
   table[row sep=crcr]{-1.30102999566398	-13.7892205911016\\
 -1.60205999132796	-12.1903836886112\\
 -1.90308998699194	-12.2728947699051\\
 -2.20411998265592	-10.9005017966681\\
 -2.50514997831991	-11.9266687401526\\
 -2.80617997398389	-9.68242556486888\\
 -3.10720996964787	-9.58648519293109\\
 };
 \addplot [color=blue,solid,forget plot]
   table[row sep=crcr]{-1.30102999566398	-1.19200819859512\\
 -1.60205999132796	-1.95811048169918\\
 -1.90308998699194	-2.24414329942003\\
 -2.20411998265592	-2.17484029621045\\
 -2.50514997831991	-2.35649789803871\\
 -2.80617997398389	-2.60730866108399\\
 -3.10720996964787	-2.8852416064269\\
 };
 \addplot [color=blue,dashed,forget plot]
   table[row sep=crcr]{-1.30102999566398	-1.49880634809607\\
 -1.60205999132796	-1.77857077963892\\
 -1.90308998699194	-2.06854240552661\\
 -2.20411998265592	-2.36393930267898\\
 -2.50514997831991	-2.6621274129253\\
 -2.80617997398389	-2.96173020544779\\
 -3.10720996964787	-3.26204644655944\\
 };
 \addplot [color=black,solid,forget plot]
   table[row sep=crcr]{-1.30102999566398	-0.387706142760986\\
 -1.60205999132796	-0.817259177902428\\
 -1.90308998699194	-1.25410673862392\\
 -2.20411998265592	-1.66710847435432\\
 -2.50514997831991	-2.04320471965376\\
 -2.80617997398389	-2.38822688971985\\
 -3.10720996964787	-2.71315344485222\\
 };
 \addplot [color=black,dashed,forget plot]
   table[row sep=crcr]{-1.30102999566398	-1.19158444469418\\
 -1.60205999132796	-1.49283529530765\\
 -1.90308998699194	-1.79393840313177\\
 -2.20411998265592	-2.09498980672939\\
 -2.50514997831991	-2.39602562492828\\
 -2.80617997398389	-2.69705709648672\\
 -3.10720996964787	-2.99808740817569\\
 };
 \end{axis}
 \end{tikzpicture}
}
 \subfigure[]{
 \begin{tikzpicture}

 \begin{axis}[%
 width=3.82222222222222cm,
 height=2.80333333333333cm,
 scale only axis,
 xmin=-3.2,
 xmax=-1.2,
 ymin=-4,
 ymax=0,
 axis x line*=bottom,
 axis y line*=left
 ]
 \addplot [color=red,solid,forget plot]
   table[row sep=crcr]{-1.30102999566398	-0.459851189384339\\
 -1.60205999132796	-0.628319380437902\\
 -1.90308998699194	-0.856517782205023\\
 -2.20411998265592	-1.12344493799844\\
 -2.50514997831991	-1.40931000860561\\
 -2.80617997398389	-1.70343823430501\\
 -3.10720996964787	-2.00121028847485\\
 };
 \addplot [color=red,dashed,forget plot]
   table[row sep=crcr]{-1.30102999566398	-0.917485144461885\\
 -1.60205999132796	-1.22578779973348\\
 -1.90308998699194	-1.53055435878329\\
 -2.20411998265592	-1.8334684904679\\
 -2.50514997831991	-2.13544357427202\\
 -2.80617997398389	-2.43694676306172\\
 -3.10720996964787	-2.73821350500683\\
 };
 \addplot [color=blue,solid,forget plot]
   table[row sep=crcr]{-1.30102999566398	-0.882474862449781\\
 -1.60205999132796	-1.3465861733012\\
 -1.90308998699194	-1.99413246165605\\
 -2.20411998265592	-3.06801778250322\\
 -2.50514997831991	-3.15277496849839\\
 -2.80617997398389	-3.19748825705331\\
 -3.10720996964787	-3.41193852531686\\
 };
 \addplot [color=blue,dashed,forget plot]
   table[row sep=crcr]{-1.30102999566398	-1.54647256273234\\
 -1.60205999132796	-1.89040862470988\\
 -1.90308998699194	-2.21254131002577\\
 -2.20411998265592	-2.52399770367845\\
 -2.50514997831991	-2.83020615886261\\
 -2.80617997398389	-3.13381633408893\\
 -3.10720996964787	-3.43613411572158\\
 };
 \addplot [color=black,solid,forget plot]
   table[row sep=crcr]{-1.30102999566398	-0.454474286127426\\
 -1.60205999132796	-0.968256361053079\\
 -1.90308998699194	-1.54364195605513\\
 -2.20411998265592	-2.14816645636776\\
 -2.50514997831991	-2.7580422329709\\
 -2.80617997398389	-3.36588521959601\\
 -3.10720996964787	-3.97130364798909\\
 };
 \addplot [color=black,dashed,forget plot]
   table[row sep=crcr]{-1.30102999566398	-1.26770687544553\\
 -1.60205999132796	-1.57421552563273\\
 -1.90308998699194	-1.87814866362423\\
 -2.20411998265592	-2.18065658859291\\
 -2.50514997831991	-2.48243072255409\\
 -2.80617997398389	-2.78383392569734\\
 -3.10720996964787	-3.08505079154388\\
 };
 \end{axis}
 \end{tikzpicture}%
}  
 \subfigure[]{
 \begin{tikzpicture}

 \begin{axis}[%
 width=3.82222222222222cm,
 height=2.80333333333333cm,
 scale only axis,
 xmin=-3.2,
 xmax=-1.2,
 ymin=-3.5,
 ymax=0,
 axis x line*=bottom,
 axis y line*=left
 ]
 \addplot [color=red,solid,forget plot]
   table[row sep=crcr]{-1.30102999566398	-0.440794885591603\\
 -1.60205999132796	-0.595374358347933\\
 -1.90308998699194	-0.809281990918945\\
 -2.20411998265592	-1.06526611383512\\
 -2.50514997831991	-1.34438327125202\\
 -2.80617997398389	-1.63480490350633\\
 -3.10720996964787	-1.93064369747154\\
 };
 \addplot [color=red,dashed,forget plot]
   table[row sep=crcr]{-1.30102999566398	-0.890877687325258\\
 -1.60205999132796	-1.19190768298584\\
 -1.90308998699194	-1.49293767863532\\
 -2.20411998265592	-1.7939676740673\\
 -2.50514997831991	-2.09499766885069\\
 -2.80617997398389	-2.39602764723562\\
 -3.10720996964787	-2.69705748478149\\
 };
 \addplot [color=blue,solid,forget plot]
   table[row sep=crcr]{-1.30102999566398	-0.735546925757686\\
 -1.60205999132796	-1.05774947213281\\
 -1.90308998699194	-1.43862298540207\\
 -2.20411998265592	-1.82388837712101\\
 -2.50514997831991	-2.18523732897041\\
 -2.80617997398389	-2.52215438669683\\
 -3.10720996964787	-2.8427218727151\\
 };
 \addplot [color=blue,dashed,forget plot]
   table[row sep=crcr]{-1.30102999566398	-1.40755653083382\\
 -1.60205999132796	-1.73276421897924\\
 -1.90308998699194	-2.04561745597023\\
 -2.20411998265592	-2.35247415613198\\
 -2.50514997831991	-2.6563945421038\\
 -2.80617997398389	-2.9588637434072\\
 -3.10720996964787	-3.26061133282256\\
 };
 \addplot [color=black,solid,forget plot]
   table[row sep=crcr]{-1.30102999566398	-0.880876589766861\\
 -1.60205999132796	-1.53597165696268\\
 -1.90308998699194	-2.45956477156951\\
 -2.20411998265592	-2.1087452049054\\
 -2.50514997831991	-2.25431331035774\\
 -2.80617997398389	-2.49271022133111\\
 -3.10720996964787	-2.76526540168023\\
 };
 \addplot [color=black,dashed,forget plot]
   table[row sep=crcr]{-1.30102999566398	-1.19135767347873\\
 -1.60205999132796	-1.49280038431052\\
 -1.90308998699194	-1.79393351927732\\
 -2.20411998265592	-2.09498915735167\\
 -2.50514997831991	-2.39602553768022\\
 -2.80617997398389	-2.69705708879654\\
 -3.10720996964787	-2.9980872106365\\
 };
 \end{axis}
 \end{tikzpicture}%

}\\
 \subfigure[]{
 \begin{tikzpicture}

 \begin{axis}[%
 width=3.82222222222222cm,
 height=2.80333333333333cm,
 scale only axis,
 xmin=-3.2,
 xmax=-1.2,
 xlabel={$\log_{10}(\delta t)$},
 ymin=-10,
 ymax=0,
 ylabel={$\log_{10}(\varepsilon_{\rm as})$},
 axis x line*=bottom,
 axis y line*=left
 ]
 \addplot [color=red,solid,forget plot]
   table[row sep=crcr]{-1.30102999566398	-0.487234028021833\\
 -1.60205999132796	-0.653296289385832\\
 -1.90308998699194	-0.874347351112273\\
 -2.20411998265592	-1.13405671305081\\
 -2.50514997831991	-1.41503944838577\\
 -2.80617997398389	-1.70640151660186\\
 -3.10720996964787	-2.00271511538662\\
 };
 \addplot [color=red,dashed,forget plot]
   table[row sep=crcr]{-1.30102999566398	-1.19190768299072\\
 -1.60205999132796	-1.49293767865505\\
 -1.90308998699194	-1.79396767432936\\
 -2.20411998265592	-2.09499766931465\\
 -2.50514997831991	-2.39602765570706\\
 -2.80617997398389	-2.69705760056586\\
 -3.10720996964787	-2.99808712976983\\
 };
 \addplot [color=blue,solid,forget plot]
   table[row sep=crcr]{-1.30102999566398	-0.886254420683424\\
 -1.60205999132796	-1.28844138821557\\
 -1.90308998699194	-1.79539151892307\\
 -2.20411998265592	-2.36327581404237\\
 -2.50514997831991	-2.95539059536133\\
 -2.80617997398389	-3.5548413785178\\
 -3.10720996964787	-4.15625045894753\\
 };
 \addplot [color=blue,dashed,forget plot]
   table[row sep=crcr]{-1.30102999566398	-2.40781069246027\\
 -1.60205999132796	-3.00703042209247\\
 -1.90308998699194	-3.60836138860682\\
 -2.20411998265592	-4.21023784257925\\
 -2.50514997831991	-4.81225027638874\\
 -2.80617997398389	-5.41432986487593\\
 -3.10720996964787	-6.01649801032371\\
 };
 \addplot [color=black,solid,forget plot]
   table[row sep=crcr]{-1.30102999566398	-0.599561145224976\\
 -1.60205999132796	-1.06343148371898\\
 -1.90308998699194	-1.59964496522918\\
 -2.20411998265592	-2.17806450872567\\
 -2.50514997831991	-2.77325442499088\\
 -2.80617997398389	-3.37351277993033\\
 -3.10720996964787	-3.97511480055253\\
 };
 \addplot [color=black,dashed,forget plot]
   table[row sep=crcr]{-1.30102999566398	-5.01116916961572\\
 -1.60205999132796	-6.15307560229286\\
 -1.90308998699194	-7.32496859516135\\
 -2.20411998265592	-8.51063281887384\\
 -2.50514997831991	-9.57008336474321\\
 -2.80617997398389	-9.52688426785064\\
 -3.10720996964787	-8.98628757740163\\
 };
 \end{axis}
 \end{tikzpicture}%

} 
\subfigure[]{
 \begin{tikzpicture}

 \begin{axis}[%
 width=3.82222222222222cm,
 height=2.80333333333333cm,
 scale only axis,
 xmin=-3.2,
 xmax=-1.2,
 xlabel={$\log_{10}(\delta t)$},
 ymin=-4.5,
 ymax=0,
 axis x line*=bottom,
 axis y line*=left
 ]
 \addplot [color=red,solid,forget plot]
   table[row sep=crcr]{-1.30102999566398	-0.454628030564337\\
 -1.60205999132796	-0.619354450780928\\
 -1.90308998699194	-0.843751881803734\\
 -2.20411998265592	-1.10781671180023\\
 -2.50514997831991	-1.39194273126039\\
 -2.80617997398389	-1.68512511068586\\
 -3.10720996964787	-1.98240659660636\\
 };
 \addplot [color=red,dashed,forget plot]
   table[row sep=crcr]{-1.30102999566398	-0.910558497212595\\
 -1.60205999132796	-1.21698278367246\\
 -1.90308998699194	-1.52079203560666\\
 -2.20411998265592	-1.82322562580012\\
 -2.50514997831991	-2.12496022451914\\
 -2.80617997398389	-2.42634307888896\\
 -3.10720996964787	-2.72754902529519\\
 };
 \addplot [color=blue,solid,forget plot]
   table[row sep=crcr]{-1.30102999566398	-0.836018097348837\\
 -1.60205999132796	-1.24775012173946\\
 -1.90308998699194	-1.76895880157889\\
 -2.20411998265592	-2.34859473896284\\
 -2.50514997831991	-2.94781621404867\\
 -2.80617997398389	-3.55102505647356\\
 -3.10720996964787	-4.15432887943651\\
 };
 \addplot [color=blue,dashed,forget plot]
   table[row sep=crcr]{-1.30102999566398	-1.50659931952066\\
 -1.60205999132796	-1.84487618167626\\
 -1.90308998699194	-2.16418749718209\\
 -2.20411998265592	-2.47424108865417\\
 -2.50514997831991	-2.77975060395114\\
 -2.80617997398389	-3.08301172488929\\
 -3.10720996964787	-3.38515220986941\\
 };
 \addplot [color=black,solid,forget plot]
   table[row sep=crcr]{-1.30102999566398	-0.548114139466529\\
 -1.60205999132796	-1.06784381367109\\
 -1.90308998699194	-1.69634375361685\\
 -2.20411998265592	-2.48338658173856\\
 -2.50514997831991	-3.97193914082816\\
 -2.80617997398389	-3.31256881623189\\
 -3.10720996964787	-3.45522822087819\\
 };
 \addplot [color=black,dashed,forget plot]
   table[row sep=crcr]{-1.30102999566398	-1.24689793932329\\
 -1.60205999132796	-1.55216441114452\\
 -1.90308998699194	-1.85540072138393\\
 -2.20411998265592	-2.15754606366737\\
 -2.50514997831991	-2.45913582241758\\
 -2.80617997398389	-2.76044597188496\\
 -3.10720996964787	-3.06161500720414\\
 };
 \end{axis}
 \end{tikzpicture}%
}
\subfigure[]{
 \begin{tikzpicture}

 \begin{axis}[%
 width=3.82222222222222cm,
 height=2.80333333333333cm,
 scale only axis,
 xmin=-3.2,
 xmax=-1.2,
 xlabel={$\log_{10}(\delta t)$},
 ymin=-4.5,
 ymax=0,
 axis x line*=bottom,
 axis y line*=left
 ]
 \addplot [color=red,solid,forget plot]
   table[row sep=crcr]{-1.30102999566398	-0.498389470401499\\
 -1.60205999132796	-0.668552864139368\\
 -1.90308998699194	-0.893139435548767\\
 -2.20411998265592	-1.15531594171713\\
 -2.50514997831991	-1.43776545413687\\
 -2.80617997398389	-1.72992359166745\\
 -3.10720996964787	-2.02665079066448\\
 };
 \addplot [color=red,dashed,forget plot]
   table[row sep=crcr]{-1.30102999566398	-1.25203634900066\\
 -1.60205999132796	-1.55478702722038\\
 -1.90308998699194	-1.85671471046011\\
 -2.20411998265592	-2.158202712931\\
 -2.50514997831991	-2.45946395830883\\
 -2.80617997398389	-2.7606100814028\\
 -3.10720996964787	-3.06169736484239\\
 };
 \addplot [color=blue,solid,forget plot]
   table[row sep=crcr]{-1.30102999566398	-0.952761164901008\\
 -1.60205999132796	-1.41560574013811\\
 -1.90308998699194	-2.07358660787378\\
 -2.20411998265592	-3.43975841252473\\
 -2.50514997831991	-3.04645111462782\\
 -2.80617997398389	-3.13876264941031\\
 -3.10720996964787	-3.36372600802077\\
 };
 \addplot [color=blue,dashed,forget plot]
   table[row sep=crcr]{-1.30102999566398	-2.70737078496654\\
 -1.60205999132796	-2.71084179626973\\
 -1.90308998699194	-2.91481221258841\\
 -2.20411998265592	-3.17430897745008\\
 -2.50514997831991	-3.45596563752542\\
 -2.80617997398389	-3.74762329745385\\
 -3.10720996964787	-4.04406975835916\\
 };
 \addplot [color=black,solid,forget plot]
   table[row sep=crcr]{-1.30102999566398	-0.504341622648382\\
 -1.60205999132796	-0.961202886770684\\
 -1.90308998699194	-1.45825085742789\\
 -2.20411998265592	-1.95231533687909\\
 -2.50514997831991	-2.40897082602943\\
 -2.80617997398389	-2.81733641664452\\
 -3.10720996964787	-3.18405717073827\\
 };
 \addplot [color=black,dashed,forget plot]
   table[row sep=crcr]{-1.30102999566398	-2.0535942959269\\
 -1.60205999132796	-2.35666803283793\\
 -1.90308998699194	-2.65856646628088\\
 -2.20411998265592	-2.95997230944742\\
 -2.50514997831991	-3.26117247851492\\
 -2.80617997398389	-3.56228348662676\\
 -3.10720996964787	-3.86335891846244\\
 };
 \end{axis}
 \end{tikzpicture}%

}
 \caption{Logarithm of the asymptotic error as a function of the logarithm
of the time step:
Scheme \#1 (solid red line), Scheme \#2 (dotted red line), Scheme \#3 (solid
blue line), Scheme \#4 (dotted blue line), Scheme \#5 (solid black line),
Scheme \#6 (dotted black line). $(\theta_f,\theta_s)=(1.0,0.0)$ (a),
$(\theta_f,\theta_s)=(0.0,0.0)$ (b), $(\theta_f,\theta_s)=(0.0,1.0)$ (c),
$(\theta_f,\theta_s)=(0.5,0.5)$ (d), $(\theta_f,\theta_s)=(0.0,0.25)$ (e)  and $(\theta_f,\theta_s)=((N+1)/(2N),0.75)$ (f).}
\label{fig:nonlinear01} 
\end{figure}
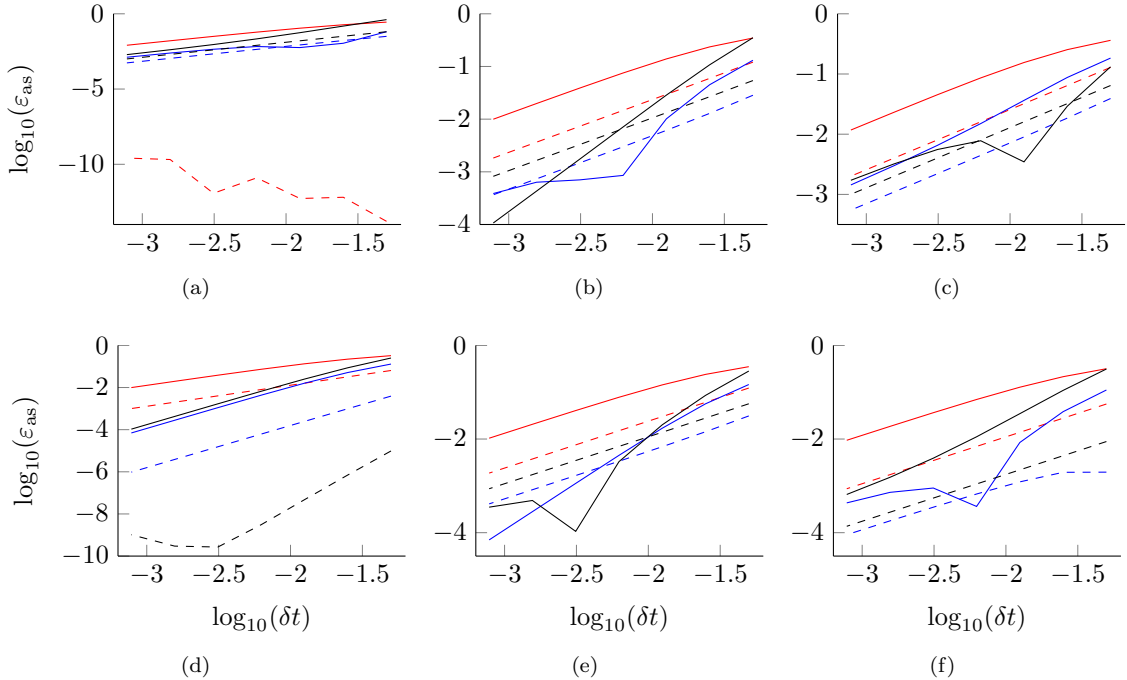

\section{Numerical analysis of the asymptotic error of splitting schemes applied to a coupled reaction-diffusion system}
\label{sec:reactiondiffusion}

{We now turn to the longtime behavior of a PDE toy-model :
a linear coupled reaction-diffusion system set over a finite space interval.
It has the property, if the boundary conditions are of homogeneous Dirichlet
type, that all its solutions asymptotically tend to zero in time,
with an exponential rate.
As we did for the linear ODE model in Section 2,
we study the approximated rate of convergence to $0$
for the solution of the problem with homogeneous Dirichlet boundary conditions
by a subcycled Lie SF method (Theorem \ref{decroissancenum}).
We then consider the non-homogeneous Dirichlet variant of the problem
and we estimate the accuracy of the asymptotic numerical state obtained with
a subcycled Lie SF method (Theorem \ref{Th:ApproxLieDirInHomo}).}

\subsection{The homogeneous Dirichlet problem}
\label{sec:homogeneousDirichlet}
 
\paragraph{The continuous problem}
This section aims at studying the behavior of time-splitting schemes
involving subcycling techniques for solving the following system
of partial differential equations
\begin{equation}
\label{systemechaleur0}
\left\{
\begin{array}{rcll}
\partial_t u & = &\nu_1 \Delta u + c_1 (v-u)\\
\partial_t v & = & \nu_2 \Delta v + c_2 (u-v) 
\end{array}\right. \qquad t>0,x\in(0,L)
,
\end{equation}

\noindent
with homogeneous Dirichlet boundary conditions at $x=0$ and $x=L$,
and given initial data $u^0$ and $v^0$ in an appropriate function space.
Here, $\Delta=\partial_x^2$ is the Laplace operator and $L>0$ is given.
Moreover, $\nu_1$ and $\nu_2$ are real positive diffusion parameters
and $c_1$ and $c_2$ are real positive reaction speed parameters.
We focus on the case where one of the equations in System \eqref{systemechaleur0}
is ``fast'' and the other is ``slow''. Moreover, we assume the ``speed''
ratios allow us to actually do subcycling. This means that
\begin{equation}
\label{ratio}
\frac{\nu_1}{\nu_2} = \frac{c_1}{c_2} = N\in\N^\star,
\end{equation}
and $N>>1$. Yet, we are not interested in the limit $N\to+\infty$.
Recall that one can expect to have similar results when only
the order of magnitude of $N$ is known (See Remark \ref{rem:Nunknown} for
the ODE system of Section \ref{sec:linearanalysis}), but we
assume that $N$ is exactly known via relation \eqref{ratio} to keep
the notations and the analysis simple.
Consequently, in accordance with Section \ref{sec:linearanalysis}, we will use
the notation $\nu=\nu_2$ and $c=c_2$.
In that case, the first equation in \eqref{systemechaleur0} is the ``fast'' one,
and the second one is the slow one since it reads
\begin{equation}
\label{systemechaleur}
\left\{
\begin{array}{rcll}
\partial_t u & = & N\nu \Delta u + N c (v-u)\\
\partial_t v & = & \nu \Delta v + c (u-v) 
\end{array}\right. \qquad t>0,x\in(0,L).
\end{equation}
Therefore, $u$ is referred to as the fast unknown and $v$ as the slow one.
With the notations introduced in Section \ref{sec:Intro},
we have
\begin{equation*}
X={\rm L}^2(0,L)^2,\quad W=\left(\begin{matrix}u\\ v\end{matrix}\right),\quad
{\mathsf f}_{\rm s} \left(\begin{matrix}u\\ v\end{matrix}\right) =
\left(\begin{matrix}0\\ \nu \Delta v + c(u-v)\end{matrix}\right),\quad
\text{and}\quad
{\mathsf f}_{\rm f} \left(\begin{matrix}u\\ v\end{matrix}\right) =
\left(\begin{matrix} \nu \Delta u + c(v-u)\\ 0\end{matrix}\right),
\end{equation*}
and the considered equations and the corresponding semigroups are linear.
Let us recall that we have the following
\begin{theorem}
\label{th:ExSol}
For all initial data $(u^0,v^0)\in {\rm L}^2(0,L)^2$,
System \eqref{systemechaleur}
has a unique solution $t\mapsto (u(t),v(t))$ in
$C^0([0,+\infty),{\rm L}^2(0,L)^2) \cap C^\infty((0,+\infty)\times[0,L],\R^2)$,
satisfying $(u,v)(0)=(u^0,v^0)$.
\end{theorem}

\begin{proof}
If one looks for solutions of the form
\begin{equation*}
  u(t,x)=\sum_{k=1}^{+\infty} \alpha_k(t)\sin\left(k\pi x/L\right)
  \quad\text{and}\quad
  v(t,x)=\sum_{k=1}^{+\infty} \beta_k(t)\sin\left(k\pi x/L\right),
\end{equation*}
then the coefficients satisfy the differential systems
\begin{equation*}
      \alpha_k' (t)  =  -N\left(c+\nu \frac{k^2\pi^2}{L^2}\right) \alpha_k (t)
        + Nc \beta_k(t),\qquad
      \beta_k' (t)  =   c \alpha_k(t) 
        -\left(c+\nu \frac{k^2\pi^2}{L^2}\right) \beta_k (t),       
\end{equation*}
and the eigenvalues $\lambda_k$ and $\mu_k$ of the matrices
$M_k=\begin{pmatrix}
 -N\left(c+\nu \frac{k^2\pi^2}{L^2}\right) &  + Nc  \\
+ c & -\left(c+\nu \frac{k^2\pi^2}{L^2}\right) \\
\end{pmatrix}$ are both real, negative and satisfy, when $k$ tends to $+\infty$,
\begin{equation*}
  \lambda_k \sim -N \nu \frac{k^2\pi^2}{L^2}
  \quad\text{and}\quad
  \mu_k \sim - \nu \frac{k^2\pi^2}{L^2}.
\end{equation*}
Existence and uniqueness of the solution of the Cauchy problem in the functional
space follow.
\end{proof}

The following theorem deals with the asymptotic behavior of the solutions of
System \eqref{systemechaleur}:
\begin{theorem}
\label{th:SolEx}
For all solutions $(u,v)$ of System \eqref{systemechaleur} and all $t\geq 0$,
we have
\begin{equation}\label{eq:taux}
\int_0^L (|u|^2+N|v|^2)(t) \dd x
\leq \Big(\int_0^L (|u|^2+N|v|^2)(0) \dd x\Big) {\rm e}^{-\frac{2\pi^2\nu}{L^2} t}.
\end{equation}
\end{theorem}

\begin{proof}
Let $(u,v)$ be a smooth solution of \eqref{systemechaleur}. We compute
\begin{align*}
\left(\frac{\dd}{\dd t} \frac{1}{2} \int_0^L (|u|^2+N|v|^2) \dd x\right)(t)
& =  N\nu \int_0^L u(t)\Delta u(t) + N\nu \int_0^L v(t)\Delta v(t)
+Nc\int_0^L (u(v-u)+v(u-v))(t) \\
& =  -N\nu \int_0^L |\nabla u(t)|^2 -\nu \int_0^LN |\nabla v(t)|^2
-Nc\int_0^L |u(t)-v(t)|^2 \\
& \leq  - \frac{2 \pi^2 \nu}{L^2 } \frac{1}{2}
\int_0^L (|u(t)|^2+N|v(t)|^2) \dd x,
\end{align*}
\noindent
using that $N\geq 1$ and  Poincar\'e's inequality.
\end{proof}

The goal of the next paragraphs is to show how this exponential convergence
to $0$ in ${\rm L}^2(0,L)$ is reproduced by splitting schemes with
(or without) subcycling.

\paragraph{The space discretization}

In the following, we will use the classical finite-difference discretization
of minus the Laplace operator,
using the symmetric tridiagonal $J\times J$ matrix
$A =toeplitz(-1,2,-1,0)$
where $J\in\N^\star$ and $\dx =L/(J+1)$.
We note for all $i\in\{0,\dots,J+1\}$, $x_i=i\cdot\dx $
and
$U=(u_1,\dots,u_J)^{\t}$ will be the solution of the discretized problem.
Let us recall that the eigenvalues and associated eigenvectors of $A$ are, for  $1\leq j \leq J$,
\begin{equation}
\label{eq:valeurspropresdeA}
\left(\lambda_j = 4
\sin^2\Big(\frac{j\pi}{2(J+1)}\Big),\big(\sin(1j\pi/(J+1)),\sin(2j\pi/(J+1)),\dots,\sin(Jj\pi/(J+1))\big)^{\t}\right).
\end{equation}
In the following, we denote by
\begin{equation}
\label{diagA}
A=ZDZ^{-1},
\end{equation}
the corresponding diagonalization of $A$.
We endow $\R^J$ with the classical Euclidian norm
\begin{equation*}
  \forall (u_1,\dots,u_J)^{\t}\in\R^J,\qquad \| (u_1,\dots,u_J)^{\t} \|_2
:=\sqrt{\frac{1}{J+1}\sum_{i=0}^J|u_i|^2}
=\sqrt{\frac{\dx}{L}\sum_{i=0}^J|u_i|^2},
\end{equation*}
with the convention that $u_0=0$ so that the norm is consistant
with the rectangle quadrature method and homogeneous
Dirichlet boundary conditions.
We use a similar definition for the Euclidian norm on $\R^J\times\R^J$,
which we also denote by $\|\cdot\|_2$.
We use the induced norms on the corresponding algebras of square matrices
which we denote by $\normetriple{\cdot}_2$.

\paragraph{The time discretization}

Assume $\dt>0$ is given.
The methods we have in mind all share the same basic idea: we discretize
in time separately the spatially-discretized versions
of both equations of System \eqref{systemechaleur}.
We consider $(p,p',q,q')\in(\N^*)^4$ such that
\begin{equation}\label{tempsphysique}
  \dfrac{q'}{q}=\dfrac{p'}{Np}.
\end{equation}
The ``fast'' one is discretized on an interval of length $\dt/(Np)$
and we denote by $\Phi_{f,\dt/(Np)}$ its numerical flow.
We iterate this method $p'$ times.
The ``slow'' one is discretized on an interval of length $\dt/q$
and we denote by $\Phi_{s,\dt/q}$ its numerical flow. We iterate this
method $q'$ times.
Then, we compute numerical flows using splitting methods and subcycling
by considering numerical flows such as
\begin{equation}
\label{LieSF}
\Psi_{{\rm Lie},\dt} = \Phi_{s,\dt} \circ \Phi_{f,\dt/N}^N,
\end{equation}
corresponding to $(p,p',q,q')=(1,N,1,1)$.
As we did in Section \ref{sec:linearanalysis} and in Section
\ref{sec:nonlinearanalysis}, we consider $\theta$-schemes
for the solution of the slow and fast equations.
We choose two parameters $(\theta_f,\theta_s)\in[0,1]^2$.
The numerical integrators involved in the splitting scheme therefore
read:
\begin{equation}\label{thetaschemeDHfast}
\Phi_{f,\dt/N}(u^n,v^n)=
\left[
\left(I+\theta_f\dt \left(c I+\nu\frac{1}{(\dx )^2} A\right)\right)^{-1}
\left(\left(I-(1-\theta_f)\dt \left(c I+\nu\frac{1}{(\dx )^2} A\right)\right)
u^n
+ c \dt v^n\right)
\, , \,
v^n
\right],
\end{equation}
\noindent
and
\begin{equation}\label{thetaschemeDHslow}
\Phi_{s,\dt}(u^n,v^n)=
\left[
u^n
\, , \,
\left(I+\theta_s\dt \left(c I+\nu\frac{1}{(\dx )^2} A\right)\right)^{-1}
\left(\left(I-(1-\theta_s)\dt \left(c I+\nu\frac{1}{(\dx )^2} A\right)\right)
v^n
+ c \dt u^n\right)
\right],
\end{equation}
where $I$ stands for the identity matrix.
This way, a stability condition reads
\begin{equation}
\label{stabcond1}
\dt \leq \frac{1}{c+4\nu/(\dx )^2}.
\end{equation}
Note also that the stability condition \eqref{stabcond1} of the scheme
is actually independent of $N$, and this is a very interesting
feature of splitting schemes involving subcycling.
Let us define for $z\in\{s,f\}$,
\begin{equation*}
B_z(\dt):=I-(1-\theta_z)\dt \left(c I+\nu\frac{1}{(\dx )^2} A\right)
\qquad{\rm and }\qquad
C_z(\dt):=I+\theta_z\dt \left(c I+\nu\frac{1}{(\dx )^2} A\right).
\end{equation*}
For the sake of simplicity, we omit the dependence in $\dt$ of $C$ and $B$,
thus noting
$(B,C)_s=(B,C)_s(\dt/q)$ and $(B,C)_f=(B,C)_f(\dt/p)$.
Since they are polynomials in $A$, the matrices
$I$, $C_s$, $C_f$, $B_s$, $B_f$, $C_s^{-1}$, $C_f^{-1}$ and $A$ do
commute for all values (distinct or not) of $\dt$. The matrices of the linear mappings $\Phi_{s,\dt/q}$ and
$\Phi_{f,\dt/(Np)}$ in the canonical basis of $\R^{2J}$ read
respectively 
\begin{equation}\label{formematDH}
M_{s}(\dt/q)=
\begin{pmatrix}
I & 0 \\
c\frac{\dt}{q} C_s^{-1} & B_s C_s^{-1}\\
\end{pmatrix}
\qquad{\rm and}\qquad
M_f(\dt/(Np))=
\begin{pmatrix}
B_f C_f^{-1} & c \frac{\dt}{p} C_f^{-1} \\
0 & I \\
\end{pmatrix}.
\end{equation}
 Let us define
$\Sigma_{z,m}=\sum_{k=0}^{m-1}(C_z^{-1}B_z)^k$ for $m\geq 1$ and $z\in\{s,f\}$.
\noindent
Therefore, the matrix of $\Phi_{f,\dt/(Np)}^{p'}$  reads
\begin{equation*}
M_f(\dt/(Np))^{p'}=
\begin{pmatrix}
(B_fC_f^{-1})^{p'} & c\frac{\dt}{p} C_f^{-1} \Sigma_{f,p'}  \\
0 & I 
\end{pmatrix}.
\end{equation*}

\noindent
Recalling \eqref{tempsphysique}, we define
$\Psi_{\dt,p,p',q,q'}=\Phi_{s,\dt/q}^{q'}\circ\Phi_{f,\dt/(Np)}^{p'}$ the matrix of
which reads
\begin{equation}
\label{eq:formematricePsi}
\begin{pmatrix}
(B_fC^{-1}_f)^{p'} & c\frac{\dt}{p} C_f^{-1} \Sigma_{f,p'} \\
c \frac{\dt}{q} C_s^{-1}(B_fC_f^{-1})^{p'}\Sigma_{s,q'} & (B_sC_s^{-1})^{q'} + c^2\frac{\dt^2}{pq} C_s^{-1}C_f^{-1}\Sigma_{s,q'}\Sigma_{f,p'}\\
\end{pmatrix}
.
\end{equation}
In particular, if $q=q'=p=1$ and $p'=N$, $\Psi_{\dt,p,p',q,q'}=\Psi_{{\rm Lie},\dt}$.

\paragraph{Rate of convergence
for the subcycled SF Lie-splitting scheme}

The central result of this subsection is the following analysis
of the rate of convergence to 0 of the numerical solutions of
Problem \eqref{systemechaleur}:

\begin{theorem}\label{decroissancenum}
Let $c,\nu>0$, $N\geq 2$. {Let us consider a subcycled SF Lie
method based on $\theta$-schemes defined by \eqref{LieSF},
\eqref{thetaschemeDHslow} and \eqref{thetaschemeDHfast}.}
Assume $J\in\N^\star$ is given.
There exists $C,\gamma,h>0$ such that for all $T>0$, all $U^0,V^0\in\R^J$,
all $\dt\in(0,h)$ and all $n\in \N$ with $n\dt\leq T$, we have
\begin{equation}
\label{eq:expdecay}
\|\Psi_{{\rm Lie},\dt}^n (U^0,V^0)\|_{2} \leq C {\rm e}^{-\gamma n\dt} 
\|(U^0,V^0)\|_{2}.
\end{equation}
One can impose $\gamma\geq N\nu\lambda_1/((N+1)(\dx )^2)$
in this case, provided $h$ is small enough.\\
The exact decay rate $\nu\pi^2/L^2$ 
from Theorem \ref{th:SolEx} \eqref{eq:taux} is of the same order as
the asymptotic numerical one $N\nu\pi^2/(L^2(N+1))$.
\end{theorem}
\begin{proof}
We perform a numerical analysis of the linear splitting method
$\Psi_{{\rm Lie},\dt}^n$. We determine its eigenvalues, show that they are real positive and control the
biggest one to obtain the exponential decay stated in \eqref{eq:expdecay}.
Let $(p,p',q,q')$ be positive integers satisfying \eqref{tempsphysique}.
Denoting by $\mathcal Z$ the matrix (see \eqref{diagA})
\begin{equation}
\label{MatriceP}
\mathcal Z = 
\begin{pmatrix}
Z & 0 \\
0& Z
\end{pmatrix},
\end{equation}

\noindent
we obtain that the matrix $\mathcal{D}:={\mathcal Z}^{-1} \Psi_{\dt,p,p',q,q'} {\mathcal Z}$ is exactly the same as that of \eqref{eq:formematricePsi} where
$A$ is replaced with $D$ in the definition of the matrices
$B_f,B_s,C_f$ and $C_s$. In particular, it consists in four square blocks,
each of size $J \times J$, each of which is diagonal. 
We infer that all the eigenvalues of $\Psi_{\dt,p,p',q,q'}$ are the roots of 
 the $J$ polynomial equations
\begin{equation}
\label{trinome}
\tau^2 - \Big((\phi^{-1}_{f}\psi_{f})^{p'} + (\phi^{-1}_{s}\psi_{s})^{q'}
+ c^2 \frac{\dt^2}{pq} \phi^{-1}_{f} \phi^{-1}_{s}\widetilde\Sigma_{s,q'}\widetilde\Sigma_{f,p'}
\Big) \tau
+ (\phi_{f}^{-1}\psi_{f})^{p'} (\phi_{s}^{-1}\psi_{s})^{q'}=0,
\end{equation}
where
\begin{align}
\label{def}
\psi_{f,s} (\mu)= 1-(1-\theta_{f,s})\dfrac{\dt}{p}\mu
\quad{\rm and}\quad
\phi_{f,s} (\mu)= 1+\theta_{f,s}\dfrac{\dt}{p}\mu,\\
\widetilde\Sigma_{f,p'}=\sum_{k=0}^{p'-1}(\phi_f^{-1}\psi_f)^k \quad{\rm and}\quad\widetilde\Sigma_{s,q'}=\sum_{k=0}^{q'-1}(\phi_s^{-1}\psi_s)^k,
\end{align}
and $\mu$ is  an eigenvalue of
$cI+\nu A/(\dx)^2$.
\noindent
We extend these six real-valued functions of $\mu$
to the continuous interval $(c,c+4\nu/(\dx)^2)$.
For $i\in\{s,f\}$,
the functions $\mu\mapsto \phi_{i}^{-1}(\mu)$ and $\mu\mapsto \psi_{i}(\mu)$
are smooth, non-increasing on $(c,c+4\nu/(\dx)^2)$ with values in $(0,1]$.
Hence, any finite product of such functions and any finite sum is
smooth and non-increasing on $(c,c+4\nu/(\dx)^2)$. Indeed,
\begin{align*}
  P:\mu\mapsto (\phi^{-1}_{f}(\mu)\psi_{f}(\mu))^{p'},\quad
  Q:\mu\mapsto(\phi^{-1}_{s}(\mu)\psi_{s}(\mu))^{q'},\quad
  \Sigma:\mu\mapsto c^2 \dfrac{\dt^2}{pq} \phi^{-1}_{f}(\mu) \phi_s^{-1}(\mu)\widetilde\Sigma_{s,q'}(\mu)\widetilde\Sigma_{f,p'}(\mu),
\end{align*}
are positive non-increasing functions on $(c,c+4\nu/(\dx)^2)$.
Note that the discriminant of the polynomial \eqref{trinome} is
\begin{align}
\mathcal{D}(\mu):=&\Big(P(\mu)+Q(\mu)+\Sigma(\mu)\Big)^2-4Q(\mu)P(\mu)\notag\\
=&\Big(Q(\mu)-P(\mu)+\Sigma(\mu)\Big)^2+4P(\mu)\Sigma(\mu)>0\label{eq:discrimQ-P}\\
=&\Big(P(\mu)-Q(\mu)+\Sigma(\mu)\Big)^2+4Q(\mu)\Sigma(\mu)>0,\label{eq:discrimP-Q}
\end{align}
\noindent
so that the eigenvalues of $\Psi_{\dt,p,p',q,q'}$ are real and can be expressed
using the functions
\begin{equation*}
\tau^-(\mu)=\frac{P(\mu)+Q(\mu)+\Sigma(\mu)-\sqrt{\mathcal{D}(\mu)}}{2}
\qquad{\rm and}\qquad
\tau^+(\mu)=\frac{P(\mu)+Q(\mu)+\Sigma(\mu)+\sqrt{\mathcal{D}(\mu)}}{2},
\end{equation*}
\noindent
for $\mu\in(c,c+4\nu/(\dx)^2)$.
Note that, with the stability condition \eqref{stabcond1},
we have $0< \tau^-(\mu)<\tau^+(\mu)$.
Moreover, we have a monotonicity property for the function
$\mu\mapsto \tau^+(\mu)$ on the interval $(c,c+4\nu/(\dx)^2)$
(see Lemma \ref{lemmemonotone}).
Hence the biggest eigenvalue of $\Psi_{\dt,p,p',q,q'}$ is
$\tau^+(\mu_1)$ with $\mu_1:=c+\nu\lambda_1/(\dx)^2$ (see~\eqref{eq:valeurspropresdeA}).

We compute an asymptotic expansion of that biggest eigenvalue when
$\dt\to 0^+$ to control the exponential decay of the
${\rm L}^2$ norm of the numerical solution provided by $\Psi_{\dt,p,p',q,q'}$.
Let $J\in\N^\star$ be fixed.
We number the eigenvalues of $cI+\nu A/\dx^2$ as follows:
\begin{equation}
  \label{eq:mui}
  \forall i\in\{1,\dots,J\},\qquad \mu_i = c+\nu \frac{\lambda_i}{\dx^2}.
\end{equation}
Since
${\phi_{f}^{-1}(\mu_1)\psi_{f}(\mu_1) =(1-(1-\theta_f)\dt \mu_1)/(1+\theta_f\dt \mu_1)}$,
we may write 
\begin{equation*}
\forall k\in\{0,\dots,p'\},\qquad(\phi_{f}^{-1}(\mu_1)\psi_{f}(\mu_1))^k = 1 - k \mu_1 \dt + {\mathcal O}(\dt^2),
\end{equation*}
We infer that
\begin{equation*}
\sum_{k=0}^{p'-1} (\phi_{f}^{-1}(\mu_1)\psi_{f}(\mu_1))^k =
p'-\mu_1\frac{p'(p'-1)}{2}\dt + {\mathcal O}(\dt^2).
\end{equation*}
We obtain Taylor expansions for $P(\mu_1)$, $Q(\mu_1)$,
$\Sigma(\mu_1)$ and then $\mathcal{D}(\mu_1)$ similarly.
Eventually, for the Lie-splitting SF method ($q=q'=p=1$ and $p'=N$),
we obtain the following Taylor expansion for $\tau^+(\mu_1)$ when $\dt$ tends
to $0$:
\begin{equation*}
\tau^+(\mu_1)=1- \gamma_0\dt+ {\mathcal O}(\dt^2),
\end{equation*}
with
\begin{equation*}
  \gamma_0:=\dfrac{(N+1)\mu_1-\sqrt{(N-1)^2\mu_1^2+4Nc^2}}{2}.
\end{equation*}
Therefore,
\begin{equation}\label{eq:tauxnum1}
  \frac{1}{\delta t}{\rm ln} (\tau^+(\mu_1)) = -\gamma_0 + {\mathcal O}(\dt).
\end{equation}
Note that, since $0<c<\mu_1$, we have $0<4Nc^2<4N\mu_1^2$. Hence
$$ (N+1)^2\mu_1^2 - (N-1)^2\mu_1^2=4N\mu_1^2>4Nc^2,$$
and $\gamma_0>0$.
Since $\tau^+(\mu_1)$ is the biggest eigenvalue of $\Psi_{{\rm Lie},\dt}$,
this proves the result.
Note also that
\begin{equation}\label{eq:tauxnum2}
\gamma_0
=\dfrac{(N+1)\mu_1-\sqrt{(N+1)^2\mu_1^2-4N(\mu_1^2-c^2)}}{2}.
\end{equation}
\noindent
Using the mean value theorem,
for some $c_{\theta}\in(0,4N(\mu_1^2-c^2))$,
we conclude that
\begin{equation*}
\gamma_0=\frac{1}{2}
\frac{1}{2}
\frac{4N(\mu_1^2-c^2)}{\sqrt{(N+1)^2\mu_1^2-c_\theta}}
>
N\frac{\mu_1^2-c^2}{(N+1)\mu_1}
=
\frac{N}{N+1}\underbrace{\frac{(\mu_1+c)}{\mu_1}}_{\geq 1}\underbrace{(\mu_1-c)}_{=\nu\lambda_1/\dx^2}
\geq
\frac{N}{N+1}\nu\frac{\lambda_1}{(\dx )^2}
.
\end{equation*}
Putting together \eqref{eq:tauxnum1} and \eqref{eq:tauxnum2} allows for the
expected choice of $\gamma$.\\
Moreover, recalling that $N$ is large and that $N\nu\lambda_1/(\dx )^2\to N\nu\frac{\pi^2}{L^2}$
as $\dx \to 0^+$ (or equivalently as $J\to+\infty$), we get the correct order of magnitude of the numerical rate of convergence.
\end{proof}

\noindent
In the proof of Theorem \ref{decroissancenum}, we used the following

\begin{lemma}
\label{lemmemonotone}
The map $\mu\mapsto \tau^+(\mu)$ is non-increasing in $(c,c+4\nu/(\dx)^2)$. Note that
$\mathcal{D}$ is not a non-increasing function of $\mu$ in general.
\end{lemma}
\begin{proof}
We use the notations of Theorem \ref{decroissancenum}.
Note that, thanks to \eqref{eq:discrimQ-P},  $\sqrt{\mathcal{D}(\mu)}>Q(\mu)-P(\mu)$ if
$Q(\mu)>P(\mu)$. Similarly, \eqref{eq:discrimP-Q} leads to
$\sqrt{\mathcal{D}(\mu)}>P(\mu)-Q(\mu)$ if $P(\mu)>Q(\mu)$ since $P$, $Q$, $\Sigma$ are
positive functions. So
 $ \sqrt{\mathcal{D}}>|P-Q|.$
Differentiating the function $\mu\mapsto \tau^+(\mu)$ with respect to $\mu$ yields
\begin{eqnarray*}
  2\sqrt{\mathcal{D}}\dfrac{\dd}{\dd \mu}\tau^+ & = &  
  (\underbrace{P'+Q'+\Sigma'}_{<0})\sqrt{\mathcal{D}} + (P+Q+\underbrace{\Sigma}_{>0})(\underbrace{P'+Q'+\Sigma'}_{<0})-2(PQ)' \\
  & < & (P'+Q'+\Sigma')|Q-P| + (P+Q)(P'+Q'+\Sigma')-2P'Q - 2PQ'\\
  & < & P'(|P-Q|+P-Q)+Q'(|Q-P|+Q-P)\\
  & \leq & 0.
\end{eqnarray*}
This implies that the derivative of $\mu\mapsto \tau^+(\mu)$ is
non-positive on $(c,c+4\nu/(\dx)^2)$ and proves the lemma. 
\end{proof}

\subsection{The inhomogeneous Dirichlet problem}

\paragraph{The continuous problem}
In this section we consider System \eqref{systemechaleur} equipped with
inhomogeneous Dirichlet boundary conditions, namely
\begin{equation}\label{nhbv}
  u(t,0)=u_l, \qquad u(t,L)=u_r, \qquad v(t,0)=v_l, \qquad v(t,L)=v_r,
\end{equation}
where $u_l,v_l,u_r$ and $v_r$ are four given real numbers.
As in the homogeneous case above (see Section \ref{sec:homogeneousDirichlet}),
there is a unique stationary solution to the boundary value problem:
\begin{proposition}
\label{prop:solsnhbv}
  The PDE system \eqref{systemechaleur} with non homogeneous Dirichlet boundary
  conditions has a unique stationary solution given by

  \begin{equation}
\begin{cases}\label{solsnhbv}
  u^\infty_{\rm ex}: & \!\!\! x\mapsto  \frac{u_l+v_l}{2} \!+\! \frac{(u_r+v_r-u_l-v_l) x}{2L} \!+\!
  \frac{(u_l-v_l)[\cosh(x/\alpha)-\cosh(L/\alpha)\sinh(x/\alpha)/\sinh(L/\alpha)]}{2}\!+\!\frac{(u_r-v_r)\sinh(x/\alpha)/\sinh(L/\alpha)}{2}\\
\\
  v^\infty_{\rm ex}: & \!\!\! x\mapsto  \frac{u_l+v_l}{2} \!+\! \frac{(u_r+v_r-u_l-v_l) x}{2L} \!-\! \frac{(u_l-v_l)[\cosh(x/\alpha)-\cosh(L/\alpha)\sinh(x/\alpha)/\sinh(L/\alpha)]}{2}\!-\!\frac{(u_r-v_r)\sinh(x/\alpha)/\sinh(L/\alpha)}{2}
\end{cases}
\end{equation}
where $\alpha=\sqrt{\nu/(2c)}$. 
\end{proposition}

\noindent
Therefore, using the linearity of the problems,
for all $(u^0,v^0)\in {\rm L}^2(0,L)^2$,
the inhomogeneous reaction-diffusion system
\eqref{systemechaleur}-\eqref{nhbv} has a unique solution
in ${C^0([0,+\infty),{\rm L}^2(0,L)^2)
\cap C^\infty((0,+\infty)\times[0,L],\R^2)}$
satisfying ${(u,v)(0)=(u^0,v^0)}$, which is obtained from that of
the homogeneous Dirichlet problem (with a modified initial datum)
by adding the constant-in-time function
\eqref{solsnhbv} to it (see Theorem \ref{th:ExSol}).
Moreover, for all initial datum $(u^0,v^0)$, the solution of the inhomogeneous
System \eqref{systemechaleur} converges exponentially fast
as $t\rightarrow +\infty$ to the stationary
solution \eqref{solsnhbv} in ${\rm L}^2(0,L)^2$.

The goal of the next paragraphs is to illustrate how well this convergence
towards (a discretized version of) the stationary solution is achieved by
numerical methods using subcycling techniques.

\paragraph{Space and time discretizations}
Using the same space discretization as above
(see Section \ref{sec:homogeneousDirichlet}),
we consider two $\theta$-schemes for the time discretization
in the spirit of what we did for the homogeneous problem
(see \eqref{thetaschemeDHfast}-\eqref{thetaschemeDHslow}),
with parameters $\theta_f$ and $\theta_s$.
Taking into account the inhomogeneous Dirichlet boundary conditions
yields a sequence $((U^n,V^n)^{\t})_{n\in\N}$ defined by an
arithmetic-geometric recursion: given $W^0=(U^0,\,V^0)^{\t}\in \R^{2J}$, we have
for all $n\geq 0$, 
\begin{equation}\label{recursion}
W^{n+1}=\mathcal{M}W^n+\mathcal{M}_u
\begin{pmatrix}
  U_{l,r}\\0_{J}
\end{pmatrix}
 +\mathcal{M}_v 
\begin{pmatrix}
  0_J\\ V_{l,r}
\end{pmatrix}=:\mathcal{M}W^n+\Upsilon
  \end{equation}
where $\mathcal{M}$ is defined as a product of matrices of the form
\eqref{formematDH}, $U_{l,r}=(u_l,0,\hdots,0,u_r)^{\t}$,
$V_{l,r}=(v_l,0,\hdots,0,v_r)^{\t}$ and $\mathcal{M}_u$ and
$\mathcal{M}_v$ are $2J$-by-$2J$ matrices, depending on $\dt$, $\dx$
and the choice of the splitting method between the two $\theta$-schemes.

Let us list the numerical experiments we conducted:
\begin{itemize}
\item Scheme \#1 (Lie - SF - slow time - subcycled): $M_s:=M_s(\dt)$ and ${M_f:=M_f(\dt/N)}$
  \begin{equation} \label{Schema01}
    \mathcal{M}=M_sM_f^N ,\,\,\, 
\mathcal{M}_u=\nu\dfrac{\dt}{\delta x^2}M_s\sum_{k=0}^{N-1}M_f^k
\begin{pmatrix}
  C_f^{-1} & 0\\ 0 & 0
\end{pmatrix}
\,\,\mbox{and }\,\, \mathcal{M}_v=\nu\dfrac{\dt}{\delta x^2}
\begin{pmatrix}
 0& 0\\ 0 & C_s^{-1} 
\end{pmatrix}
  \end{equation}
\item Scheme \#2 (Lie - SF - fast time - no subcycling): $M_s:=M_s(\dt/N)$ and $M_f:=M_f(\dt/N)$
  \begin{equation*}
    \mathcal{M}=M_sM_f ,\,\,\, 
\mathcal{M}_u=\nu\dfrac{\dt}{\delta x^2}M_s
\begin{pmatrix}
  C_f^{-1} & 0\\ 0 & 0
\end{pmatrix}
\,\,\mbox{and }\,\, \mathcal{M}_v=\frac{\nu}{N}\dfrac{\dt}{\delta x^2}
\begin{pmatrix}
 0& 0\\ 0 & C_s^{-1} 
\end{pmatrix}
  \end{equation*}
\item Scheme \#3 (Strang - SFS - slow time - subcycled): $M_s:=M_s(\dt/2)$ and $M_f:=M_f(\dt/N)$
  \begin{equation*}
    \mathcal{M}=M_sM_f^NM_s ,\,\,\, 
\mathcal{M}_u=\nu\dfrac{\dt}{\delta x^2}M_s\sum_{k=0}^{N-1}M_f^k
\begin{pmatrix}
  C_f^{-1} & 0\\ 0 & 0
\end{pmatrix}
\,\,\mbox{and }\,\, \mathcal{M}_v=\nu \frac{\dt}{2\delta x^2} (I_{2J}+M_sM_f^N) 
\begin{pmatrix}
 0& 0\\ 0 & C_s^{-1} 
\end{pmatrix}
  \end{equation*}
\item Scheme \#4 (Strang - SFS - fast time - no subcycling): $M_s:=M_s(\dt/(2N))$ and $M_f:=M_f(\dt/N)$
  \begin{equation*}
    \mathcal{M}=M_sM_fM_s ,\,\,\, 
\mathcal{M}_u=\nu\dfrac{\dt}{\delta x^2}M_s
\begin{pmatrix}
  C_f^{-1} & 0\\ 0 & 0
\end{pmatrix}
\,\,\mbox{and }\,\, \mathcal{M}_v=\nu\dfrac{\dt}{2N\delta x^2}
(I_{2J}+M_sM_f) 
\begin{pmatrix}
 0& 0\\ 0 & C_s^{-1} 
\end{pmatrix}
  \end{equation*}
\item Scheme \#5 (weighted - slow time - subcycled): $M_s:=M_s(\dt)$ and $M_f:=M_f(\dt/N)$
  \begin{equation*}
    \mathcal{M}=\frac{1}{2}(M_s M_f^N + M_f^N M_s ),\,\,\, 
  \end{equation*}
  \begin{equation*}
\mathcal{M}_u=\nu\dfrac{\dt}{\delta x^2}\frac{I_{2J}+M_s}{2}\sum_{k=0}^{N-1}M_f^k
\begin{pmatrix}
  C_f^{-1} & 0\\ 0 & 0
\end{pmatrix}
\,\,\mbox{and }\,\,
\mathcal{M}_v=
\nu\dfrac{\dt}{\delta x^2}(I_{2J}+M_f^N)
\begin{pmatrix}
 0& 0\\ 0 & C_s^{-1} 
\end{pmatrix}
  \end{equation*}
\item Scheme \#6 (weighted - fast time - no subcycling): $M_s:=M_s(\dt/N)$
and $M_f:=M_f(\dt/N)$
  \begin{equation*}
    \mathcal{M}=\frac{1}{2}(M_s M_f + M_f M_s ),\,\,\, 
\mathcal{M}_u=\nu\dfrac{\dt}{\delta x^2}\frac{I_{2J}+M_s}{2}
\begin{pmatrix}
  C_f^{-1} & 0\\ 0 & 0
\end{pmatrix}
\,\,\mbox{and }\,\, \mathcal{M}_v=\frac{\nu}{N}\dfrac{\dt}{\delta x^2}
\frac{I_{2J}+M_f}{2}
\begin{pmatrix}
 0& 0\\ 0 & C_s^{-1} 
\end{pmatrix}
  \end{equation*}
\end{itemize}
In order to keep notations short, we used the following convention.
For the subcycled schemes (Schemes \#1, \#3 and \#5), an application
of the iteration formula \eqref{recursion} corresponds to a time interval
of length $\dt$.
However, for the schemes \#2, \#4 and \#6, an application of the iteration
formula \eqref{recursion} corresponds to a time interval of length $\dt/N$.
Note that, in particular, this convention does not modify the asymptotic
states of the methods (meaning that if $W^\infty_{\rm num}\in\R^{2J}$ is an
asymptotic for the iteration of Scheme \#2 (resp. \#4, resp. \#6),
then it is also an asymptotic state for Scheme \#2 (resp. \#4, resp. \#6)
iterated $N$ times).

\paragraph{Equilibrium states of the splitting schemes}

We prove the existence of a unique equilibrium state for the splitting
Scheme \#1 above, comment on the rate of convergence of the scheme
towards its equilibrium state and also analyze how close the 
equilibrium state of the scheme is to a projection on the numerical
space grid of the equilibrium state \eqref{solsnhbv} of the continuous
reaction-diffusion system \eqref{systemechaleur} with inhomogeneous
Dirichlet conditions \eqref{nhbv} in an ${\rm L}^2$ sense.
Following \eqref{stabcond1}, we denote by $\CFL(J)$ the positive real number
$$\CFL(J)=\frac{1}{c+4\nu/\dx^2}=\frac{1}{c+4\nu(J+1)^2/L^2}.$$
To compute asymptotic
numerical solutions of a given method of type \eqref{recursion},
we need to solve the $2J$-by-$2J$ linear system
\begin{equation}\label{eq:ls}
  (I_{2J}-\mathcal{M})W= \Upsilon.
\end{equation}
\begin{proposition}
Let $\dt,\dx>0$ satisfying \eqref{stabcond1} be fixed.
For a subcycled SF Lie-splitting method of the form
\eqref{Schema01}, {based on $\theta$-schemes,} there exists a unique numerical asymptotic state
defined as the unique solution $W^\infty_{\rm num}$
of the linear system \eqref{eq:ls}.
\end{proposition}
\begin{proof}
Since $\dt,\delta x$ satisfy \eqref{stabcond1}, we know from
Theorem \ref{decroissancenum} that the spectral radius of the matrix
$\mathcal M$ of $\Psi_{{\rm Lie},\dt}$ in the canonical basis of $\R^{2J}$ is
less than 1.
Hence, the matrix $I_{2J}-\mathcal M$ is invertible and the
numerical asymptotic state is well-defined and unique.
\end{proof}

Using the linearity of the problems, we infer that the numerical
rate of convergence towards this asymptotic state is then given
by Theorem~\ref{decroissancenum}.

Let us state and prove the central result of this section, {\em i.e.} the convergent asymptotic
behavior of the subcycled SF Lie method (Scheme \#1)
{involving $\theta$-schemes}:
\begin{theorem}
\label{Th:ApproxLieDirInHomo}
Provided that $\dt\in(0,\CFL(J))$, the asymptotic state of Scheme \#1 {(subcycled Lie method based on $\theta$-schemes)} is a
{\em uniform-in-$\dt$} second order approximation of the exact asymptotic state given
in Proposition \ref{prop:solsnhbv}:
  \begin{equation*}
     \begin{pmatrix}
    \Pi_\dx \big(u^\infty_{\rm ex}\big) \\
    \Pi_\dx \big(v^\infty_{\rm ex}\big) \\
  \end{pmatrix}
  -
  W^\infty_{\rm num}(\dt)
  = {\mathcal O}(\dx^2),
\end{equation*}
where for $w\in C^0([0,L])$, $\Pi_{\dx}(w)=(w(x_1),\hdots,w(x_J))^{\t}$.
\end{theorem}
\begin{proof}
  To analyze the asymptotic convergence of Scheme \#1, we put the projections
  $\Pi_\dx \big(u^\infty_{\rm ex}\big)$ and $\Pi_\dx \big(v^\infty_{\rm
    ex}\big)$ of the exact solutions $u^\infty_{\rm ex}$ and $v^\infty_{\rm ex}$
  defined in \eqref{solsnhbv} in the numerical scheme.  Using the identity
  \begin{equation*}
    \frac{1}{\dx^2} A \Pi_\dx \big(u^\infty_{\rm ex}\big) =
    - \Pi_\dx \big(\Delta u^\infty_{\rm ex}\big)
    + U_{l,r}
    + \mathcal O(\dx^2),
  \end{equation*}
  and the fact that $(u^\infty_{\rm ex},v^\infty_{\rm ex})$
  is an equilibrium state of problem \eqref{systemechaleur}
  with the inhomogeneous Dirichlet
  boundary conditions \eqref{nhbv}, we first compute
  \begin{equation*}
    M_f 
    \begin{pmatrix}
      \Pi_\dx \big(u^\infty_{\rm ex}\big) \\
      \Pi_\dx \big(v^\infty_{\rm ex}\big) \\
    \end{pmatrix}
    =
    \begin{pmatrix}
      \Pi_\dx \big(u^\infty_{\rm ex}\big) \\
      \Pi_\dx \big(v^\infty_{\rm ex}\big) \\
    \end{pmatrix}
    - \nu \frac{\dt}{\dx^2} 
    \begin{pmatrix}
      C_f^{-1} U_{l,r}\\
      0 \\
    \end{pmatrix}
    +{\mathcal O} (\dt (\dx)^2),
  \end{equation*}
  \noindent where the constant in the $\mathcal O$ is independent of $\delta t$
  and $\delta x$ provided that the CFL condition is fulfilled. Iterating this
  computation, we obtain
  \begin{equation}\label{eq:consistency}
    M_f^N
    \begin{pmatrix}
      \Pi_\dx \big(u^\infty_{\rm ex}\big) \\
      \Pi_\dx \big(v^\infty_{\rm ex}\big) \\
    \end{pmatrix}
    =
    \begin{pmatrix}
      \Pi_\dx \big(u^\infty_{\rm ex}\big) \\
      \Pi_\dx \big(v^\infty_{\rm ex}\big)\\
    \end{pmatrix}
    - \nu \frac{\dt}{\dx^2} \sum_{k=0}^{N-1} M_f^{k}
    \begin{pmatrix}
      C_f^{-1} & 0\\
      0 & 0 \\
    \end{pmatrix}
    \begin{pmatrix}
      U_{l,r}\\
      0 \\
    \end{pmatrix}
    +{\mathcal O} (\dt (\dx)^2),
  \end{equation}
  where, once again, the constant in the $\mathcal O$ is independent of $\dt$
  and
  $\dx$ provided that the CFL condition \eqref{stabcond1} is fulfilled.
  This is due to the fact that we have
  \begin{equation*}
    M_f {\mathcal O}(\dt (\dx)^2) = {\mathcal O}(\dt (\dx)^2),
  \end{equation*}
  provided that $\dt\in(0,\CFL(J))$ thanks to Lemma \ref{Lemmebornitude} (see Appendix), which
  gives uniform estimates of   $\normetriple{M_{s,f}}_2$.
  Multiplying \eqref{eq:consistency} by $M_s$ and using again
that $(u^\infty_{\rm ex},v^\infty_{\rm ex})$
is an equilibrium state of problem \eqref{systemechaleur}
with the inhomogeneous Dirichlet
boundary conditions \eqref{nhbv},
we finally get
\begin{equation*}
  (I_{2J}- M_s M_f^N)
  \begin{pmatrix}
    \Pi_\dx \big(u^\infty_{\rm ex}\big) \\
    \Pi_\dx \big(v^\infty_{\rm ex}\big) \\
  \end{pmatrix}
  =
  \nu \frac{\dt}{\dx^2} M_s \sum_{k=0}^{N-1} M_f^{k}
  \begin{pmatrix}
    C_f^{-1} & 0\\
    0 & 0 \\
  \end{pmatrix}
  \begin{pmatrix}
    U_{l,r}\\
    0 \\
  \end{pmatrix}
  +\nu \frac{\dt}{\dx^2}
  \begin{pmatrix}
    0 & 0\\
    0 & C_s^{-1} \\
  \end{pmatrix}
  \begin{pmatrix}
    0\\
    V_{l,r} \\
  \end{pmatrix}
  +{\mathcal O} (\dt (\dx)^2).
\end{equation*}
Comparing this relation with that defining the numerical equilibrium state
\eqref{recursion} (with the right-hand side defined in \eqref{Schema01}), we
infer that
\begin{equation}
\label{eq:madifferenceamoi}
  (I_{2J}-M_sM_f^N)
  \left(
    \begin{pmatrix}
      \Pi_\dx \big(u^\infty_{\rm ex}\big) \\
      \Pi_\dx \big(v^\infty_{\rm ex}\big) \\
    \end{pmatrix}
    -
    W^\infty_{\rm num}
  \right)
  = \dt {\mathcal O}(\dx^2),
\end{equation}
where the constant in the $\mathcal O$ is independent of $\dt$ and $\dx$
provided that the CFL condition \eqref{stabcond1} is fulfilled.
Finally, we can use the result of Proposition \ref{ControleInverse} (see Appendix)
which states that there exists a constant $C>0$ such that for all
$\dt$ and $\dx$ satisfying the CFL condition, we have
$\normetriple{(I-M_s M_f^N)^{-1}}_2 \leq \frac{C}{\dt}$.
This estimate together with that written in \eqref{eq:madifferenceamoi}
proves the result.
\end{proof} 
\paragraph{Numerical tests}
The numerical tests we conducted for several values of $\theta_f$, $\theta_s$
and $N$ showed numerically
that the matrix $I_{2J}-\mathcal{M}$ is also invertible for
Schemes \#3 and \#4.  We
show here the graphs obtained with Scheme \#1 for the following sets of
parameters, $N=10$ being fixed, $\nu_1=c_1=1.0$, $L=2\pi$,
$J=20,40,80,160$, $\dx=L/(J+1)$:
\begin{itemize}
\item $(u_l,u_r,v_l,v_r)=(1,2,-1,4)$, 
  $\dt=\dx^2/\nu_1/2$, $(\theta_f,\theta_s)=(0,0)$ [explicit,explicit]
\item $(u_l,u_r,v_l,v_r)=(2,4,-1,4)$,
  $\dt=0.01$, $(\theta_f,\theta_s)=(1/2,1/2)$ [Crank-Nicolson,Crank-Nicolson]
\end{itemize}
From Figure \ref{fig:asDNH}, we see that
the asymptotic error has the behavior predicted by Theorem
\ref{Th:ApproxLieDirInHomo} no matter the values of $\theta_f$ and
$\theta_s$: the numerical order is close to $2$ in $\dx$ (provided
the CFL condition is fulfilled).

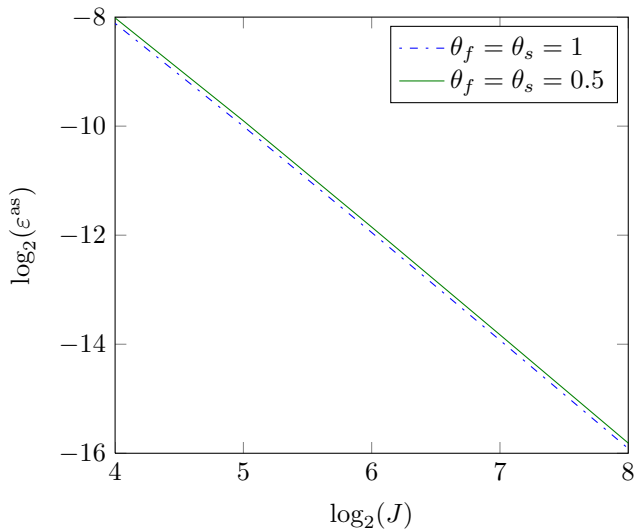
\begin{figure}[!h] 
  \centering
 \begin{tikzpicture}

 \begin{axis}[%
 width=6.82222222222222cm,
 height=5.80333333333333cm,
 scale only axis,
 xmin=4,
 xmax=8,
 xlabel={$\log_2(J)$},
 ymin=-16,
 ymax=-8,
 ylabel={$\log_2(\varepsilon^{\rm as})$},
 legend style={draw=black,fill=white,legend cell align=left}
 ]
 \addplot [color=blue,dash pattern=on 1pt off 3pt on 3pt off 3pt]
   table[row sep=crcr]{4	-8.1151048577972\\
 5	-10.0018663228\\
 6	-11.950857740424\\
 7	-13.9270671770892\\
 8	-15.9136466613269\\
 };
 \addlegendentry{$\theta_f=\theta_s=1$};

 \addplot [color=black!50!green]
   table[row sep=crcr]{4	-8.0151048577972\\
 5	-9.90186632280002\\
 6	-11.850857740424\\
 7	-13.8270671770892\\
 8	-15.8136466613269\\
 };
 \addlegendentry{$\theta_f=\theta_s=0.5$};

 \end{axis}
 \end{tikzpicture}%

  \caption{${\rm L}^\infty$-error of the asymptotic numerical and exact states for
    explicit/explicit and Crank-Nicolson/Crank-Nicolson schemes. The numerical
    order is $1.95$. {We chose these two cases because, for the ODE toy-problem, the A-orders were respectively 1 and 2. There is no visible difference for the PDE problem.}
}
  \label{fig:asDNH}
\end{figure}

 \section{Conclusion and perspectives}
Speeding up computations through a subcycling procedure is widely used, but the
asymptotic behavior of the numerical solution in large time is a
concern.
Indeed, there are two limits involved, as $\dt$ (and $\dx$ in the PDE
case) tend to $0$ and as the final time $T$ tends to $+\infty$. We
proved for an illustrative case of ODE systems that the asymptotic error is
at least of the same order of
convergence as the local-in-time error, and can even be better since
there exists 
  combinations of (local) first order schemes that lead to second
  asymptotic order ! The analysis of the convergence rate of the subcycled
  scheme has been performed for ODE and PDE toy-models, showing that the Strang
  splitting associated with Crank-Nicolson schemes was the only way to get a
  second order approximation of the exact rate. Finally, in the case of a
  coupled reaction-diffusion system with inhomogeneous Dirichlet boundary
  conditions, we were able to prove
  that the asymptotic numerical solution obtained through a subcycled scheme is
  a uniform-in-$\dt$ second order approximation in $\dx$ of the exact asymptotic
  state.\\ 
  {Our aim is now to tackle much more difficult and general cases, such as
    a fully
  coupled hyperbolic-parabolic system. The level of complexity is a lot higher in such cases,
  since the ratio of the characteristic times of the different phenomena, which 
  we modelled in the present paper as a constant $N$, cannot be defined at the continuous
  level, because of the speed of propagation of a hyperbolic equation is finite whereas the speed of 
  propagation of a parabolic equation is infinite. At the discrete level, the ratio will appear
  in the CFL conditions ($\dt=O(\dx$) for the hyperbolic equation and $\dt=O(\dx^2)$ for the parabolic
  equation) : the subcycling techniques can provide schemes which are a lot more efficient than traditional
  splitting schemes, allowing for a CFL $\dt=O(\dx)$ if the parabolic equation is subcycled. The analysis will however 
  be intricate since $N$ is related to $1/{\dx}$. }

\section*{Aknowledgment}

The authors would like to thank warmly
the referees for their highly valuable comments
which helped improve the quality of this paper.

\clearpage

\appendix

\section{FS to SF computations}
\label{sec:FS_to_SF}

Let us define the matrix $$\Pi:=
\begin{pmatrix}
  0&1\\1&0
\end{pmatrix},
$$
and let us denote by $G[\alpha,\beta]$ a matrix of the form
\eqref{formeGdeltat}.
Let $A$ be a 2-by-2 matrix.
Then $\Pi\, A$ exchanges the lines of $A$ and $A\,\Pi$
exchanges the columns. Thus, if $\lambda\in\R$,
\begin{equation*}
  \Pi\, M_s(\lambda)\,\Pi=M_f(\lambda),
\end{equation*}
and, if $\alpha,\beta\in (0,1)$,
\begin{equation*}
  \Pi\, G[\alpha,\beta]\, \Pi=G[\beta,\alpha].
\end{equation*}
 Since $\Pi^2=I$, it means that $M_s(\lambda)$ and $M_f(\lambda)$ are similar,
 thus share the same spectrum.   
In Section~\ref{sec:linearanalysis}, we computed the A-orders and rates of
convergence of SF (fast, then slow) and FSF (fast, then slow, then fast) type
schemes. We show here that the results we obtained can easily be applied to FS and SFS schemes.
\paragraph{Lie-splitting schemes}
  Consider  $\lambda_s, \lambda_f\in(0,1)$.
According to Lemma~\ref{boris} and Remark \ref{Galphabeta}, we define
$\alpha(\lambda_s,\lambda_f)$ and $\beta(\lambda_s,\lambda_f)$ as
\begin{equation*}
M_s(\lambda_s)M_f(\lambda_f)=G[\alpha(\lambda_s,\lambda_f),\beta(\lambda_s,\lambda_f)].
\end{equation*}
Since
\begin{equation*}
M_f(\lambda_f)M_s(\lambda_s)=\Pi\, M_s(\lambda_f)M_f(\lambda_s)\,\Pi,
\end{equation*}
we infer that
\begin{equation*}
  M_f(\lambda_f)M_s(\lambda_s)=\Pi\,
  G[\beta(\lambda_f,\lambda_s),\alpha(\lambda_f,\lambda_s)] \,\Pi .
\end{equation*}
Consequently, we can deduce  the convergence rate and the A-order of the FS
methods at once from the results we obtained for the SF methods.
\paragraph{Strang-splitting methods} In the same way, knowing 
$M_f(\lambda_f)M_s(\lambda_s)M_f(\lambda_f)$, one can deduce the
convergence rate and the A-order of $M_f(\lambda_f)M_s(\lambda_s)M_f(\lambda_f)$
by noting that
\begin{equation*}
  M_s(\lambda_s)M_f(\lambda_f)M_s(\lambda_s)=\Pi\, M_f(\lambda_s)M_s(\lambda_f)M_f(\lambda_s)\,\Pi.
\end{equation*}
\section{Helpful estimates for the proof of Theorem \ref{Th:ApproxLieDirInHomo}}
The following lemma is helpful for the proof of Theorem
\ref{Th:ApproxLieDirInHomo}.
  \begin{lemma}
    \label{Lemmebornitude}
    For all positive $c,\mu,L$,
there exists a positive constant $C>0$ such that, for all $J\in\N^\star$ and
    all $\delta t\in (0,{\rm CFL}(J))$, we have
    \begin{equation*}
      \normetriple{M_s}_2\leq C \qquad {\rm and} \qquad
      \normetriple{M_f}_2\leq C.
    \end{equation*}
  \end{lemma}
  \begin{remark}
    Note that the constant $C$ above is in fact greater than $1$, even if the
    matrices have their spectrum in the interval $[0,1]$. This is due to the
    lack of symmetry in those matrices.
  
  \end{remark}

\begin{proof}
  Since the situation for $M_s$ and $M_f$ is very similar, we prove the
  inequality for $M_f$ only, and we start with the decomposition
  \begin{equation*}
    M_f =
    \begin{pmatrix}
      C_f^{-1} & 0 \\
      0 & I_J \\
    \end{pmatrix}
    \begin{pmatrix}
      B_f & c \dt I_J \\
      0 & I_J \\
    \end{pmatrix}.
  \end{equation*}
  Recall that for any square matrix $R$ with real
  coefficients, $\normetriple{R}_2^2 = \rho(R^{\t} R)$,
  where $\rho$ denotes the spectral radius.
  The CFL condition \eqref{stabcond1} ensures that the spectrum of $C_f^{-1}$
  lies in $(0,1]$. Since the first matrix in the product above is symmetric, we
  infer that its norm is $\sqrt{\rho(I_J)}=1$.
  Hence, using the algebra property for $\normetriple{\cdot}_2$,
  it is sufficient
  to prove the result for the second matrix in the product above, which is {\it
    not} symmetric. We are left with the computation of
  the eigenvalues of the symmetric non-negative matrix
  \begin{equation*}
    \begin{pmatrix}
      B_f^2 & c \dt B_f \\
      c\dt B_f & (1 + c^2 \dt^2) I_J\\
    \end{pmatrix},
  \end{equation*}
  the eigenvalues of which are the $2J$ roots of the $J$ polynomials
  \begin{equation*}
    X^2-(\mu_p^2 + (1+c^2\dt^2))X + \mu_p^2, \qquad 1\leq p \leq J,
  \end{equation*}
  where $(\mu_p)_{1\leq p\leq J}$ denotes the list of the eigenvalues of $B_f$.
  The CFL condition \eqref{stabcond1} ensures that
  for all ${p\in\{1,\dots,J\}}$, $\mu_p\in [0,1]$. Hence, the
  greatest eigenvalue of the corresponding polynomial above is less than
  $2(1+1+c^2\dt^2)$.
  Moreover, the CFL condition also provides us with an estimate on $\dt$
  which yields the result with ${C=\sqrt{2(2+c^2/(c+16\nu/L^2)^2)}}$.
\end{proof}

One can control the inverse of the matrix of System \eqref{eq:ls}
by the following proposition to prove Theorem \ref{Th:ApproxLieDirInHomo}.
\begin{proposition}
  \label{ControleInverse}
  There exists a positive constant $C>0$ such that for all $J\in \N^\star$ and
  all $\dt\in(0,\CFL(J))$,
  \begin{equation}
    \label{EstimationInverse}
    \normetriple{(I-M_s M_f^N)^{-1}}_2 \leq \frac{C}{\dt}.
  \end{equation}
\end{proposition}

\begin{proof}
  Let us fix $J\in \N^\star$ and $\dt\in(0,\CFL(J))$.  Using the conjugation
  with the orthogonal matrix $\mathcal Z$ (see \eqref{MatriceP}), we have that
  the $\normetriple{\cdot}_2$-norm of $I_{2J}-M_s M_f^N$ is equal to that of the
  same matrix where $A$ is replaced with $D$ (see \eqref{diagA}). The latter
  matrix has a very particular structure: the four $J$-by-$J$ matrices defining
  it are diagonal.  Let us denote by $(a_i)_{1\leq i\leq J}$, $(b_i)_{1\leq
    i\leq J}$, $(c_i)_{1\leq i\leq J}$, and $(d_i)_{1\leq i\leq J}$ these
  entries such that
  \begin{equation*}
    \zeta:={\mathcal Z}^{-1} (I-M_s M_f^N) {\mathcal Z} = 
    \begin{pmatrix}
      \begin{matrix}
        a_1 & 0 & 0 \\
        0   & \ddots & 0 \\
        0   & 0 & a_J\\
      \end{matrix}
      &
      \begin{matrix}
        b_1 & 0 & 0 \\
        0   & \ddots & 0 \\
        0   & 0 & b_J\\
      \end{matrix}
      \\
      \begin{matrix}
        c_1 & 0 & 0 \\
        0   & \ddots & 0 \\
        0   & 0 & c_J\\
      \end{matrix}
      &
      \begin{matrix}
        d_1 & 0 & 0 \\
        0   & \ddots & 0 \\
        0   & 0 & d_J\\
      \end{matrix}
      \\
    \end{pmatrix}.
  \end{equation*}
  The eigenvalues of $\zeta$ lie in $(0,1)$ (see Theorem
  \ref{decroissancenum}). Hence, $\zeta$ is invertible and its inverse
is given by
  \begin{equation*}
    \zeta^{-1}= {\mathcal Z}^{-1} (I-M_s M_f^N)^{-1} {\mathcal Z} = 
    \begin{pmatrix}
      \begin{matrix}
        \alpha_1 & 0 & 0 \\
        0   & \ddots & 0 \\
        0   & 0 & \alpha_J\\
      \end{matrix}
      &
      \begin{matrix}
        \beta_1 & 0 & 0 \\
        0   & \ddots & 0 \\
        0   & 0 & \beta_J\\
      \end{matrix}
      \\
      \begin{matrix}
        \gamma_1 & 0 & 0 \\
        0   & \ddots & 0 \\
        0   & 0 & \gamma_J\\
      \end{matrix}
      &
      \begin{matrix}
        \delta_1 & 0 & 0 \\
        0   & \ddots & 0 \\
        0   & 0 & \delta_J\\
      \end{matrix}
      \\
    \end{pmatrix},
  \end{equation*}
  where for all $i\in\{1,\dots,J\}$,
  \begin{equation*}
    \begin{pmatrix}
      a_i & b_i \\
      c_i & d_i \\
    \end{pmatrix}^{-1} =
    \begin{pmatrix}
      \alpha_i & \beta_i \\
      \gamma_i & \delta_i \\
    \end{pmatrix}=:m_i.
  \end{equation*}
  One can check easily that
  \begin{equation*}
    \normetriple{\zeta^{-1}}_2
    = {\max}_{1\leq i\leq J} \normetriple{m_i}_2.
  \end{equation*}
  Moreover, we have
  \begin{equation*}
    \normetriple{m_i}_2^2 = \frac{a_i^2 + b_i^2 + c_i^2 + d_i^2 +
      \sqrt{(a_i^2 + b_i^2 + c_i^2 + d_i^2)^2-4(a_id_i-b_i c_i)^2}}{2(a_id_i-b_i c_i)^2}
    \leq
    \frac{a_i^2 + b_i^2 + c_i^2 + d_i^2}{(a_id_i-b_i c_i)^2}.
  \end{equation*}
  We split the upper bound above as follows
  \begin{equation}
    \label{EstimNorme2}
    \normetriple{m_i}_2^2 \leq
    \frac{b_i^2 + c_i^2}{(a_id_i-b_i c_i)^2}
    +
    \frac{a_i^2 + d_i^2}{(a_id_i-b_i c_i)^2},
  \end{equation}
  and we prove an estimate of the form ${\mathcal O}(1/\dt^2)$ for the two terms
  in the sum above.  In view of \eqref{eq:formematricePsi}, we have
  \begin{equation*}
    a_i=1-P(\mu_i), \
    b_i=-c\dt(\phi_f^{-1}\tilde\Sigma_{f,N})(\mu_i), \
  \end{equation*}
  and
  \begin{equation*}
    c_i=-c\dt(\phi_s^{-1}P)(\mu_i)\ {\rm and} \
    d_i=1-Q(\mu_i) - c^2\dt^2(\phi_s^{-1}\phi_f^{-1}\tilde\Sigma_{f,N})(\mu_i),
  \end{equation*}
  where the $\mu_i$ are defined by \eqref{eq:mui} as the ordered eigenvalues
  of $cI+\nu A/\dx^2$.
  For the first term in the upper bound \eqref{EstimNorme2}, let us show that
  the numerator is ${\mathcal O}(\dt^2)$ while the denominator is bounded from
  below by a positive constant times $\dt^4$.\\
  On the one hand, we have
  \begin{equation}
    \label{numerateur}
    |b_i|^2 \leq c^2 N^2 \dt^2 \qquad \text{and} \qquad
    |c_i|^2 \leq c^2 \dt^2.
  \end{equation}
  On the other hand, for all $i\in\{1,\dots,J\}$, we have
  \begin{align*}
    a_i d_i-b_i c_i & = (1-P(\mu_i))(1-Q(\mu_i))
    - c^2\dt^2 (\phi_s^{-1}\phi_f^{-1}\tilde \Sigma_{f,N})(\mu_i)\\
    & = \bigg(1-(\psi_f\phi_f^{-1})^N(\mu_i)\bigg)(1-Q(\mu_i)) - c^2 \dt^2 \bigg(\phi_s^{-1} \phi_f^{-1} \frac{1-(\psi_f\phi_f^{-1})^N}{1-\psi_f\phi_f^{-1}}\bigg)(\mu_i)\\
    & = \left(\bigg(\frac{1-(\psi_f\phi_f^{-1})^N}{\phi_s}\bigg) \bigg(\phi_s -
      \psi_s-\frac{c^2\dt^2}{\phi_f-\psi_f}\bigg)\right)(\mu_i).
  \end{align*}
  The CFL condition \eqref{stabcond1} ensures that $\dt\mu_i$,
  $\psi_s(\mu_i),\phi_s^{-1}(\mu_i),\psi_f(\mu_i)$, $\phi_f^{-1}(\mu_i)$ and
  $P(\mu_i)$ belong to $(0,1]$.  In view of the definitions (\ref{def}), we have
  \begin{equation*}
    (\phi_s-\psi_s)(\mu_i)=\dt\mu_i=(\phi_f-\psi_f)(\mu_i),
  \end{equation*}
  so that
  \begin{equation}
    \label{determinant}
    a_i d_i-b_i c_i = \dt \frac{(1-(\psi_f\phi_f^{-1})^N(\mu_i))}{\phi_s(\mu_i)}
    \frac{\mu_i^2-c^2}{\mu_i}.
  \end{equation}
  The CFL condition \eqref{stabcond1} implies that $1/\phi_s(\mu_i)\geq 1/2$ and
  \begin{equation*}
    0<(\psi_f\phi_f^{-1})^N(\mu_i)\leq(\psi_f\phi_f^{-1})(\mu_i)
    = \frac{1-(1-\theta_f)\dt\mu_i}{1+\theta_f\dt\mu_i}.
  \end{equation*}
  Therefore, we have
  \begin{equation}
    \label{estimegeom}
    1-(\psi_f\phi_f^{-1})^N(\mu_i)\geq 1-(\psi_f\phi_f^{-1})(\mu_i)
    = \frac{\dt\mu_i}{1+\theta_f\dt\mu_i} \geq \frac{\dt\mu_i}{2}.
  \end{equation}
  This allows to bound from below
  \begin{equation*}
    a_i d_i-b_i c_i\geq \frac{\dt^2}{4} (\underbrace{\mu_i+c}_{\geq c})(\underbrace{\mu_i-c}_{=\nu\lambda_i/\dx^2})
    \geq c\nu\frac{\dt^2}{4}\frac{\lambda_1}{\dx^2}.
  \end{equation*}
  Recall that for all $x\in(0,\pi/2)$, $\sin(x)\geq 2x/\pi$, so that
  \begin{equation}
    \label{estimconcave}
    \frac{\lambda_1}{\dx^2} = \frac{4}{\dx^2}
    \sin^2\Big(\frac{\pi}{2}\frac{1}{(J+1)}\Big)
    \geq 4 \frac{(J+1)^2}{L^2} \frac{4}{\pi^2} \frac{\pi^2}{4} \frac{1}{(J+1)^2}
    \geq \frac{4}{L^2}.
  \end{equation}
  This proves
  \begin{equation}
    \label{denominateur}
    a_i d_i-b_i c_i\geq \frac{c\nu}{L^2} \dt^2.
  \end{equation}
  Using \eqref{numerateur} and \eqref{denominateur}, there exists a positive
  constant $C$ such that
  \begin{equation}
    \label{interm1}
    \forall J\in\N^\star,\quad \forall \dt\in(0,\CFL(J)),\qquad
    \frac{b_i^2 + c_i^2}{(a_id_i-b_i c_i)^2} \leq \frac{C}{\dt^2}.
  \end{equation}
  Let us now bound the second term in the right hand side of
  \eqref{EstimNorme2}.  Let us fix $J\in\N^\star$ and $i\in (0,\CFL(J))$ again.
  From \eqref{determinant}, we have
  \begin{equation*}
    \frac{1}{(a_i d_i - c_i b_i)^2} =
    \frac{1}{\dt^2} \frac{\phi_s^2(\mu_i)}{(1-(\psi_f\phi_f^{-1})^N(\mu_i))^2}
    \bigg(\frac{\mu_i}{\mu_i^2-c^2}\bigg)^2.
  \end{equation*}
  A similar direct calculation yields
  \begin{eqnarray*}
    a_i^2+d_i^2 & = & \bigg(1-(\psi_f\phi_f^{-1})^N(\mu_i)\bigg)^2
    +\bigg(\frac{\phi_s(\mu_i)-\psi_s(\mu_i)}{\phi_s(\mu_i)}
    -c^2\dt^2\frac{1}{\phi_s\phi_f(\mu_i)}\frac{1-(\psi_f\phi_f^{-1})^N(\mu_i)}{1-\psi_f\phi_f^{-1}(\mu_i)}\bigg)^2 \\
    & = & \bigg(1-(\psi_f\phi_f^{-1})^N(\mu_i)\bigg)^2
    \bigg[1+\frac{1}{\phi_s^2(\mu_i)}\bigg(\frac{\mu_i \dt}{1-(\psi_f\phi_f^{-1})^N(\mu_i)}
    -c^2\dt^2\frac{1}{\phi_f(\mu_i)-\psi_f(\mu_i)}\bigg)^2\bigg]\\
    & = & \bigg(1-(\psi_f\phi_f^{-1})^N(\mu_i)\bigg)^2
    \bigg[1+\frac{1}{\phi_s^2(\mu_i)}\bigg(\frac{\mu_i \dt}{1-(\psi_f\phi_f^{-1})^N(\mu_i)}
    -\frac{c^2}{\mu_i}\dt\bigg)^2\bigg].\\
  \end{eqnarray*}
  We infer
  \begin{equation}
    \label{ecritsecondterme}
    \frac{a_i^2+d_i^2}{(a_i d_i - c_i b_i)^2} =
    \frac{1}{\dt^2} \phi_s^2(\mu_i)
    \bigg(\frac{\mu_i}{\mu_i^2-c^2}\bigg)^2
    \bigg[1+\frac{1}{\phi_s^2(\mu_i)}\bigg(\frac{\mu_i \dt}{1-(\psi_f\phi_f^{-1})^N(\mu_i)}
    -\frac{c^2}{\mu_i}\dt\bigg)^2\bigg].
  \end{equation}
  We can bound the terms in the product above as follows. The CFL condition
  \eqref{stabcond1} implies that $\phi_s^2(\mu_i)\leq 4$.  Moreover, using
  \eqref{estimconcave}, we have
  \begin{equation*}
    \frac{\mu_i}{\mu_i^2-c^2} = \underbrace{\frac{\mu_i}{(\mu_i+c)}}_{\leq 1}
    \frac{1}{(\mu_i-c)}\leq \frac{\dx^2}{\nu \lambda_i}
    \leq \frac{\dx^2}{\nu \lambda_1}
    \leq \frac{L^2}{4 \nu}.
  \end{equation*}
  Recall that $1/\phi_s(\mu_i)^2\leq 1$.  From \eqref{estimegeom}, we obtain
$\mu_i \dt/(1-(\psi_f\phi_f^{-1})^N(\mu_i))\leq 2$.
  For the last term in the product, we have
  \begin{equation*}
    \frac{c^2}{\mu_i}\dt = \underbrace{c\dt}_{\leq 1} \underbrace{\frac{c}{c+\nu\lambda_i/\dx^2}}_{\leq 1} \leq 1.
  \end{equation*}
  Using these inequalities in \eqref{ecritsecondterme}, taking products and
  using Young's inequality, we infer that
  \begin{equation}
    \label{interm2}
    \forall J\in\N^\star,\quad \forall \dt\in(0,\CFL(J)),\qquad
    \frac{a_i^2 + d_i^2}{(a_id_i-b_i c_i)^2} \leq \frac{11}{4}\frac{L^4}{\nu^2}\frac{1}{\dt^2}.
  \end{equation}
  The inequalities \eqref{interm1} and \eqref{interm2} together with
  \eqref{EstimNorme2} prove the result.
\end{proof}

\bibliographystyle{alpha}
\bibliography{split}

\newcommand{\etalchar}[1]{$^{#1}$}
\begin{thebibliography}{BGM{\etalchar{+}}88}

\bibitem[ADBN08]{abn}
D.~Aregba-Driollet, M.~Briani, and R.~Natalini.
\newblock Asymptotic high-order schemes for $2\times2$ dissipative hyperbolic
  systems.
\newblock {\em SIAM Journal on Numerical Analysis}, 46(2):869--894, 2008.

\bibitem[BGM{\etalchar{+}}88]{bgmps}
M.~O. Bristeau, R.~Glowinski, B.~Mantel, J.~Periaux, and G.~{S.} Singh.
\newblock On the use of subcycling for solving the compressible
  {N}avier-{S}tokes equations by operator-splitting and finite element methods.
\newblock {\em Communications in Applied Numerical Methods}, 4(3):309--317,
  1988.

\bibitem[BNO{\etalchar{+}}14]{bryan2014enzo}
Greg~L Bryan, Michael~L Norman, Brian~W O'Shea, Tom Abel, John~H Wise,
  Matthew~J Turk, Daniel~R Reynolds, David~C Collins, Peng Wang, Samuel~W
  Skillman, et~al.
\newblock Enzo: An adaptive mesh refinement code for astrophysics.
\newblock {\em The Astrophysical Journal Supplement Series}, 211(2):19, 2014.

\bibitem[CF08]{CS2008}
P.~Csom\'os and I.~Farag\'o.
\newblock Error analysis of the numerical solution of split differential
  equations.
\newblock {\em Mathematical and Computer Modelling}, 48(7--8):1090 -- 1106,
  2008.

\bibitem[CFH05]{Csomos05}
P.~Csom{\'o}s, I.~Farag{\'o}, and {\'A}.~Havasi.
\newblock Weighted sequential splittings and their analysis.
\newblock {\em Computers \& Mathematics with Applications}, 50(7):1017--1031,
  2005.

\bibitem[CGL08]{cgl}
J.~A. Carrillo, T.~Goudon, and P.~Lafitte.
\newblock Simulation of fluid and particles flows: asymptotic preserving
  schemes for bubbling and flowing regimes.
\newblock {\em Journal of Computational Physics}, 227(16):7929--7951, 2008.

\bibitem[CGLV08]{cglv}
J.~A. Carrillo, T.~Goudon, P.~Lafitte, and F.~Vecil.
\newblock Numerical schemes of diffusion asymptotics and moment closures for
  kinetic equations.
\newblock {\em Journal of Scientific Computing}, 36(1):113--149, 2008.

\bibitem[CL71]{CL71}
M.~G. Crandall and T.~M. Liggett.
\newblock Generation of semi-groups of nonlinear transformations on general
  banach spaces.
\newblock {\em American Journal of Mathematics}, 93(2):265 -- 298, 1971.

\bibitem[Dan98]{wd1}
W.J.T. Daniel.
\newblock A study of the stability of subcycling algorithms in structural
  dynamics.
\newblock {\em Computer Methods in Applied Mechanics and Engineering},
  156(1--4):1 -- 13, 1998.

\bibitem[Dan03]{wd2}
W.J.T. Daniel.
\newblock A partial velocity approach to subcycling structural dynamics.
\newblock {\em Computer Methods in Applied Mechanics and Engineering}, 192:375
  -- 394, 2003.

\bibitem[DG09]{mg1}
J.~Diaz and M.~{J.} Grote.
\newblock Energy conserving explicit local time stepping for second-order wave
  equations.
\newblock {\em SIAM Journal on Scientific Computing}, 31(3):1985--2014, 2009.

\bibitem[GLG05]{glg}
P.~Godillon-Lafitte and T.~Goudon.
\newblock A coupled model for radiative transfer: {D}oppler effects,
  equilibrium, and nonequilibrium diffusion asymptotics.
\newblock {\em Multiscale Modeling \& Simulation. A SIAM Interdisciplinary
  Journal}, 4(4):1245--1279 (electronic), 2005.

\bibitem[GM10]{mg2}
M.~{J.} Grote and T.~Mitkova.
\newblock Explicit local time-stepping methods for {M}axwell's equations.
\newblock {\em Journal of Computational and Applied Mathematics},
  234(12):3283--3302, 2010.

\bibitem[GM13]{mg3}
M.~J. Grote and T.~Mitkova.
\newblock High-order explicit local time-stepping methods for damped wave
  equations.
\newblock {\em Journal of Computational and Applied Mathematics}, 239:270--289,
  2013.

\bibitem[HW04]{hw}
E.~Hairer and G.~Wanner.
\newblock {\em Solving ordinary differential equations II: Stiff and
  differential-algebraic problems}, volume~2.
\newblock Springer, 2004.

\bibitem[Jin99]{j1999}
S.~Jin.
\newblock Efficient asymptotic-preserving ({AP}) schemes for some multiscale
  kinetic equations.
\newblock {\em SIAM Journal on Scientific Computing}, 21(2):441--454, 1999.

\bibitem[Jin10]{j2010}
S.~Jin.
\newblock Asymptotic preserving ({AP}) schemes for multiscale kinetic and
  hyperbolic equations: a review.
\newblock {\em Lecture Notes for Summer School on ''Methods and Models of
  Kinetic Theory'' (M\&MKT), Porto Ercole (Grosseto, Italy)}, 2010.

\bibitem[LM08]{lm}
M.~Lemou and L.~Mieussens.
\newblock A new asymptotic preserving scheme based on micro-macro formulation
  for linear kinetic equations in the diffusion limit.
\newblock {\em SIAM Journal on Scientific Computing}, 31(1):334--368, 2008.

\bibitem[McL02]{mll}
R.~I. McLachlan.
\newblock Families of high-order composition methods.
\newblock {\em Numerical Algorithms}, 31(1-4):233--246, 2002.

\bibitem[Pip97]{p}
S.~Piperno.
\newblock Explicit/implicit fluid/structure staggered procedures with a
  structural predictor and fluid subcycling for 2d inviscid aeroelastic
  simulations.
\newblock {\em International Journal for Numerical Methods in Fluids},
  25(10):1207--1226, 1997.

\bibitem[TEHW10]{taylor2010subcycled}
Mark~A Taylor, Katherine~J Evans, James~J Hack, and Pat Worley.
\newblock Subcycled dynamics in the spectral community atmosphere model version
  4.
\newblock {\em Proc. SciDAC 2010}, 2010.

\bibitem[Tem96]{temam}
R.~Temam.
\newblock Multilevel methods for the simulation of turbulence. {A} simple
  model.
\newblock {\em Journal of Computational Physics}, 127(2):309--315, 1996.

\end{thebibliography}

\end{document}